\documentclass[12pt]{amsart}
\usepackage{amsmath,latexsym,amsfonts,amssymb,amsthm}
\usepackage{geometry}
\usepackage{hyperref}
\usepackage{mathrsfs}
\usepackage{graphicx,color}
\usepackage{tikz-cd}
\usepackage[all,cmtip]{xy}
\usepackage{enumitem}
\usepackage{bbm}

\newtheorem{prop}{Proposition}
\newtheorem{thm}[prop]{Theorem}
\newtheorem{cor}[prop]{Corollary}
\newtheorem{conj}[prop]{Conjecture}
\newtheorem{lem}[prop]{Lemma}

\theoremstyle{definition}
\newtheorem{que}[prop]{Question}
\newtheorem{defn}[prop]{Definition}

\newtheorem{rem}[prop]{\it Remark}

\numberwithin{equation}{section}

\newcommand{\bC}{\mathbb{C}}
\newcommand{\bR}{\mathbb{R}}
\newcommand{\bA}{\mathbb{A}}
\newcommand{\bQ}{\mathbb{Q}}
\newcommand{\bZ}{\mathbb{Z}}

\newcommand{\bN}{\mathbb{N}}
\newcommand{\bk}{\mathbbm{k}}
\newcommand{\bK}{\mathbb{K}}

\newcommand{\cX}{\mathcal{X}}
\newcommand{\cY}{\mathcal{Y}}

\newcommand{\cO}{\mathcal{O}}
\newcommand{\cL}{\mathcal{L}}
\newcommand{\cI}{\mathcal{I}}

\newcommand{\cJ}{\mathcal{J}}
\newcommand{\cE}{\mathcal{E}}

\newcommand{\cD}{\mathcal{D}}
\newcommand{\cH}{\mathcal{H}}

\newcommand{\cS}{\mathcal{S}}

\newcommand{\cZ}{\mathcal{Z}}

\newcommand{\fa}{\mathfrak{a}}
\newcommand{\fb}{\mathfrak{b}}
\newcommand{\fc}{\mathfrak{c}}
\newcommand{\fm}{\mathfrak{m}}

\newcommand{\Spec}{\mathrm{Spec}}
\newcommand{\Supp}{\mathrm{Supp}}
\newcommand{\lct}{\mathrm{lct}}

\newcommand{\hvol}{\widehat{\mathrm{vol}}}

\newcommand{\ord}{\mathrm{ord}}

\newcommand{\Val}{\mathrm{Val}}

\newcommand{\e}{\mathrm{e}}

\newcommand{\vol}{\mathrm{vol}}
\newcommand{\depth}{\mathrm{depth}}
\newcommand{\rom}[1]{\lowercase\expandafter{\romannumeral #1\relax}}

\DeclareMathOperator{\Jac}{Jac}

\newcommand{\ab}{\mathfrak{a}_\bullet}

\newcommand{\cXt}{{\cX_t}}
\newcommand{\cXgt}{{\cX_{\gt}}}
\newcommand{\gt}{\overline{t}}
\newcommand{\cDt}{\cD_t}
\newcommand{\cDgt}{\cD_{\gt}}
\newcommand{\st}{\sigma(t)}
\newcommand{\sgt}{\sigma(\gt)}

\DeclareMathOperator{\hs}{e}
\DeclareMathOperator{\sing}{sing}
\DeclareMathOperator{\nvol}{\widehat{vol}}
\DeclareMathOperator{\HS}{H}
\DeclareMathOperator{\gr}{gr}
\DeclareMathOperator{\Mod}{Mod}
\DeclareMathOperator{\Hilb}{Hilb}

\begin{document}

\title[normalized volume of a singularity is lower semicontinuous]
{The normalized volume of a singularity is lower semicontinuous}
\author{Harold Blum}

\address{Department of Mathematics\\
  University of Michigan\\
  Ann Arbor, MI 48109--1043\\
  USA}
\email{blum@umich.edu}

\author{Yuchen Liu}
\address{Department of Mathematics, Yale University, New Haven, CT 06511, USA.}
\email{yuchen.liu@yale.edu}

\date{\today}

\begin{abstract}
We show that in any $\bQ$-Gorenstein flat family
of klt singularities, normalized volumes
are lower semicontinuous with respect to the Zariski topology.
A quick consequence is that smooth points 
have the largest normalized volume among 
all klt singularities. Using an alternative
characterization of K-semistability developed
by Li, Liu and Xu, we show that K-semistability
is a very generic or empty condition in any $\bQ$-Gorenstein
flat family of log Fano pairs.
\end{abstract}

\maketitle

\section{Introduction}

Given an $n$-dimensional complex klt singularity $(x\in (X,D))$, Chi Li \cite{Li15a} 
introduced the \emph{normalized volume} function on the space $\Val_{x,X}$ of real valuations of $\bC(X)$ centered at $x$. 
More precisely, for any such valuation $v$, its normalized volume
is defined as $\hvol_{x,(X,D)}(v):=A_{X,D}(v)^n\vol(v)$, where $A_{X,D}(v)$
is the log discrepancy of $v$ with respect to $(X,D)$ according to \cite{JM12, BdFFU15}, and 
$\vol(v)$ is the volume of $v$ according to \cite{ELS03}. 
Then we can define the \emph{normalized volume of a klt singularity}
$(x\in (X,D))$ by 
\[
 \hvol(x,X,D):=\min_{v\in\Val_{x,X}}\hvol_{x,(X,D)}(v)
\]
where the existence of minimizer of $\hvol$ was
shown recently in \cite{Blu16}. We also denote
$\hvol(x,X):=\hvol(x,X,0)$.

The normalized volume of a klt singularity
$x\in (X,D)$ carries interesting information of its geometry 
and topology. It was shown by the second author and Xu 
that $\hvol(x,X,D)\leq n^n$ and equality holds
if and only if $(x\in X\setminus \Supp(D))$ is
smooth (see \cite[Theorem A.4]{LiuX17} or Theorem \ref{maxhvol}).
By \cite{Xu14} the local algebraic fundamental group $\hat{\pi}_1^{\mathrm{loc}}(X,x)$
of a klt singularity $x\in X$ is always finite.
Moreover, assuming the conjectural finite degree
formula of normalized volumes \cite[Conjecture 4.1]{LiuX17}, the size of $\hat{\pi}_1^{\mathrm{loc}}(X,x)$ is bounded from above by 
$n^n/\hvol(x,X)$ (see Remark \ref{r_localpi1}).
If $X$ is a Gromov-Hausdorff
limit of K\"ahler-Einstein Fano 
manifolds, then Li and Xu \cite{LX17} showed that $\hvol(x,X)=n^n\cdot\Theta(x,X)$ 
where $\Theta(x,X)$ is the volume density 
of a closed point $x\in X$ (see \cite{HS16, SS17} for 
background materials).
\medskip

In this article, it is shown that the normalized volume of a singularity  is lower semicontinuous in families.

\begin{thm}\label{mainthm}
Let $\pi:(\cX,\cD)\to T$ together with a section $\sigma: T\to \cX$ be a
 $\bQ$-Gorenstein flat family of complex klt singularities over a normal variety $T$.
Then the function $t\mapsto\hvol(\sigma(t),\cX_t,\cD_t)$
on $T(\bC)$ is lower semicontinuous with respect
to the Zariski topology.
\end{thm}

One quick consequence of Theorem \ref{mainthm} is that smooth points
have the largest normalized volumes among all klt singularities
(see Theorem \ref{maxhvol} or \cite[Theorem A.4]{LiuX17}). Another natural consequence is that if $X$ is a Gromov-Hausdorff limit of K\"ahler-Einstein Fano manifolds, then the volume density function $x\mapsto\Theta(x,X)$ on $X(\bC)$ is lower semicontinuous in the Zariski topology, which is stronger than being lower semicontinuous in the Euclidean topology mentioned in \cite{SS17} (see Corollary \ref{cor:sctheta}).

We also state the following
natural conjecture on constructibility of normalized volumes of klt 
singularities (see also \cite[Conjecture 4.11]{Xu17}).

\begin{conj}\label{mainconj}
Let $\pi:(\cX,\cD)\to T$ together with a section $\sigma: T\to \cX$ be a
 $\bQ$-Gorenstein flat family of complex klt singularities over a normal
 variety $T$. Then the function $t\mapsto\hvol(\sigma(t),\cX_{t},\cD_t)$
 on $T(\bC)$ is constructible.
\end{conj}

\medskip

Verifying the Zariski openness of K-semistability is 
an important step in the construction of an algebraic moduli space
of K-polystable $\bQ$-Fano varieties. In a smooth family of Fano manifolds,
Odaka \cite{Oda13} and Donaldson \cite{Don15} showed that the locus of fibers admitting K\"ahler-Einstein metrics (or equivalently, being K-polystable) with
discrete automorphism groups is Zariski open.  This was
generalized by Li, Wang and Xu \cite{LWX14} where they proved the
Zariski openness of K-semistability in a $\bQ$-Gorenstein
flat families of smoothable $\bQ$-Fano varieties
in their construction of the proper moduli space of smoothable
K-polystable $\bQ$-Fano varieties (see \cite{SSY16, Oda15} for
related results).
A common feature is that analytic methods were used essentially in proving these results.

Using the alternative characterization of K-semistability
by the affine cone construction 
developed by Li, the second author, and Xu in \cite{Li15b, LL16, LX16}, 
we apply Theorem \ref{mainthm} to prove the following result on weak openness
of K-semistability. Unlike the results described in the previous paragraph, our result is proved using purely algebraic methods and hence can be applied to $\bQ$-Fano families with non-smoothable fibers (or more generally, families of log Fano pairs).

\begin{thm}\label{openk}
Let $\varphi:(\cY,\cE)\to T$ be a $\bQ$-Gorenstein flat 
family of complex log Fano pairs over a normal base $T$.
If $(\cY_o,\cE_o)$ is log K-semistable 
for some closed point $o\in T$, then the following statements hold:
\begin{enumerate}
 \item There exists an intersection $U$ of countably many Zariski
 open neighborhoods of $o$, such that $(\cY_t,\cE_t)$ is
 log K-semistable for any closed point $t\in U$.
 In particular, $(\cY_t,\cE_t)$ is log K-semistable for a very general 
 closed point $t\in T$.
 \item Denote by $\eta$ the generic point of $T$, then
 the geometric generic fiber $(\cY_{\bar{\eta}},\cE_{\bar{\eta}})$
 is log K-semistable.
 \item Assume Conjecture \ref{mainconj} is true, then such $U$
 from (1) can be chosen as a genuine Zariski open neighborhood of $o$.
\end{enumerate}
\end{thm}

The following corollary generalizes \cite[Theorem 4]{Li17} and follows easily from Theorem \ref{openk}. Note that a similar result for Fano cones is proved by Li and Xu independently in \cite[Proposition 2.36]{LX17}.
\begin{cor}\label{specialdeg}
 Suppose a complex log Fano pair $(Y,E)$ specially degenerates to a log K-semistable
 log Fano pair $(Y_0,E_0)$, then $(Y,E)$ is also log K-semistable.
\end{cor}

Our strategy to prove Theorem \ref{mainthm} is to study invariants
of ideals instead of invariants of valuations. From Liu's characterization
of normalized volume by normalized multiplicities of ideals (see
\cite[Theorem 27]{Liu16} or Theorem \ref{eqhvol}), we know
\[
\hvol(\sigma(t),\cX_t,\cD_t)=\inf_{\fa}\lct(\cX_t,\cD_t;\fa)^n\cdot\e(\fa)\]
where the infimum is taken over all ideals $\fa \subset \cO_{\cX_t}$ cosupported at
$\sigma(t)$. These ideals are parametrized by a relative 
Hilbert scheme of $\cX/T$ with countably many components. Clearly $\fa\mapsto\lct(\cX_t,\cD_t;\fa)$ is lower semicontinuous
on the Hilbert scheme,  but $\fa\mapsto\e(\fa)$ may only be upper semicontinuous. Thus, it is unclear  what semicontinuity properties $\fa \mapsto \lct(\fa)^n\cdot\e(\fa)$ may have.

To fix this issue, we introduce the \emph{normalized colength of singularities}
$\widehat{\ell_{c,k}}(\sigma(t),\cX_t,\cD_t)$
by taking the infimum of $\lct(\cX_t,\cD_t;\fa)^n\cdot\ell(\cO_{\sigma(t),\cX_t}/\fa)$
for ideals $\fa$ satisfying $\fa\supset \fm_{\sigma(t)}^k$ and $\ell(\cO_{\sigma(t),\cX_t}/\fa)\geq ck^n$. The normalized colength function
behaves better in families since the colength function $\fa\mapsto\ell(\cO_{\sigma(t),\cX_t}/\fa)$
is always locally constant in the Hilbert scheme, so  $\fa\mapsto\lct(\cX_t,\cD_t;\fa)^n\cdot\ell(\cO_{\sigma(t),\cX_t}/\fa)$
is constructibly lower semicontinuous on the Hilbert scheme. 
Thus, the properness of Hilbert schemes implies that $t\mapsto
\widehat{\ell_{c,k}}(\sigma(t),\cX_t,\cD_t)$ is constructibly lower
semicontinuous on $T$.
Then we prove a key equality between the asymptotic normalized 
colength $\widehat{\ell_{c,\infty}}(\sigma(t),\cX_t,\cD_t)$ 
and the normalized volume $\hvol(\sigma(t),\cX_t,\cD_t)$ when $c$ is small (see Theorem \ref{hvolcolength})  using
local Newton-Okounkov bodies following \cite{Cut13, KK14} (see
Lemma \ref{colengthmult}) and convex geometry (see Appendix \ref{app:lattice}).
Then by establishing a uniform approximation of volumes by colengths
(see Theorem \ref{thm:volcol}) and generalizing Li's Izumi and properness estimates
\cite{Li15a} to families (see Theorems \ref{thm:izumifamily} and \ref{thm:properness}),
we show that the normalized colengh functions uniformly
approximate the normalized volume function from above (see Theorem \ref{thm:convncol}).
Putting these ingredients together, we get the proof of Theorem \ref{mainthm}.
\medskip

This paper is organized as follows. In Section \ref{prelim}, we give
the preliminaries including notations, normalized volumes of singularities, and $\bQ$-Gorenstein flat
families of klt pairs. In Section \ref{hatl},
we introduce the concept of normalized colengths of singularities.
We show in Theorem \ref{hvolcolength} that the normalized volume
of a klt singularity is the same as its asymptotic normalized colength.
The proof of Theorem \ref{hvolcolength} uses a comparison
of colengths and multiplicities established in Lemma \ref{colengthmult}. In Section \ref{sec_fieldext},
we study the normalized volumes and normalized colength after
algebraically closed field extensions. 
In Section \ref{sec:unifapp}, we establish a uniform approximation
of volume of a valuation by colengths of its valuation ideals.
In Section \ref{sec:izumi}, we generalize Li's Izumi and properness estimates
to families. The results from Sections \ref{sec:unifapp} and \ref{sec:izumi}
enable us to prove the uniform approximation of normalized volumes
by normalized colengths from above in families (see Section \ref{sec:hvolcol}).
The proofs of main theorems
are presented in Section \ref{sec_proofs}. We give applications
of our main theorems in Section \ref{sec_appl}. Theorem \ref{maxhvol}
generalizes the inequality part of \cite[Theorem A.4]{LiuX17}.
We show that the volume density function on a Gromov-Hausdorff limit of K\"ahler-Einstein manifolds is lower semicontinuous in the Zariski topology (see Corollary \ref{cor:sctheta}).
We give an effective upper bound on the degree of finite quasi-\'etale maps 
over klt singularities on Gromov-Hausdorff limits of K\"ahler-Einstein
Fano manifolds (see Theorem \ref{ghlimit}). In Appendix \ref{app:lattice}
we provide certain convex geometric results on lattice points counting that are needed
in proving Lemma \ref{colengthmult}. In Appendix \ref{app:HS}, we 
provide results on constructbility of Hilbert--Samuel funtions that
are needed in proving uniform approximation results in Section \ref{sec:unifapp}.
\medskip

\noindent \emph{Postscript}:
After this document was first posted on the arXiv, 
the authors went on to show that the global log canonical threshold and the stability threshold are lower semicontinuous in families of polarized varieties \cite{BL18}. 
The results in \emph{loc. cit.} may be viewed as global analogues of Theorem 1 and their proofs are similar in spirit (though, the technical details are quite different).

\subsection*{Acknowledgements} The first author would like to thank his advisor Mircea Musta\c{t}\u{a} for numerous useful discussions and his constant support. In addition, he would like to thank Mattias Jonsson, Ilya Smirnov, and Tommaso de Fernex for many useful conversations.  The second author would like to thank  Chi Li and Chenyang Xu for fruitful discussions.  

We wish to
thank J\'anos Koll\'ar, Linquan Ma, Sam Payne,
Xiaowei Wang, Ziquan Zhuang, and the anonymous referees for helpful
comments on this document. We are also grateful to Ruixiang Zhang for his help
on Proposition \ref{convexgeo}.

\section{Preliminaries}\label{prelim}
\subsection{Notations}
In this paper, all varieties are assumed to be irreducible, reduced, and defined over a (not necessarily 
algebraically closed) field $\bk$ of characteristic $0$. 
For a variety $T$ over $\bk$, we denote the residue field of
any scheme-theoretic point $t\in T$ by $\kappa(t)$.
Given a morphism $\pi:\cX\to T$ between varieties over $\bk$, we write $\cX_t:=\cX\times_{T}\Spec(\kappa(t))$ for the scheme theoretic fiber over $t\in T$. We also denote
the geometric fiber of $\pi$ over $t\in T$ by $\cX_{\gt}:=\cX\times_{T}\Spec(\overline{\kappa(t)})$.
Suppose $X$ is a variety over $\bk$ and $x\in X$ is a $\bk$-rational point. Then for any
field extension $\bK/\bk$, we denote $(x_{\bK}, X_{\bK}):=(x,X)\times_{\Spec(\bk)}\Spec(\bK)$.

Let $X$ be a normal variety over $\bk$ and $D$ be an effective $\bQ$-divisor on $X$. We say that $(X,D)$ is a \emph{Kawamata log terminal (klt) pair} if $(K_X+D)$ is $\bQ$-Cartier and  $K_Y-f^*(K_X+D)$ has coefficients $>-1$ on some log resolution $f:Y\to (X,D)$. A klt pair $(X,D)$ is called \emph{a log Fano pair} if in addition $X$ is proper and $-(K_X+D)$ is ample. A klt pair $(X,D)$ together with a closed point $x\in X$ is called a \emph{klt singularity} $(x\in (X,D))$.

Let $(X,D)$ be a klt pair. For an ideal sheaf $\fa$ on $X$, we define
the \emph{log canonical threshold of $\fa$ with respect to $(X,D)$} by 
\[
\lct(X,D;\fa):=\inf_E\frac{A_{X,D}(\ord_E)}{\ord_E(\fa)},
\]
where the infimum is taken over all prime divisors $E$ on a log resolution $f:Y\to (X,D)$. We will often use the notation $\lct(\fa)$ to abbreviate $\lct(X,D;\fa)$ once the klt pair $(X,D)$ is specified.
If $\fa$ is co-supported at a single closed point $x\in X$, we define the \emph{Hilbert--Samuel multiplicity} of $\fa$ as
\[
\e(\fa):=\lim_{m\to\infty}\frac{\ell(\cO_{x,X}/\fa^m)}{m^n/n!}
\]
where $n:=\dim(X)$ and $\ell(\cO_{x,X}/\fa^m)$ denotes the length of $\cO_{x,X}/\fa^m$ as an $\cO_{x,X}$-module.

\subsection{Valuations}
Let $X$ be a variety defined over a field $\bk$ and $x\in X$ closed point. By a valuation of the function field $K(X)$, we mean a valuation $v\colon K(X)^\times \to \bR$ that is trivial on $\bk$. 
By convention, we set $v(0):=+\infty$. 
Such a valuation $v$ has \emph{center} $x$ if  $v$ is $\ge 0$ on $\cO_{x,X}$ and $>0$ on the maximal ideal of $\cO_{x,X}$. We write $\Val_{x,X}$ for the set of valuations of $K(X)$ with center $x$.

To any valuation $v\in \Val_{x,X}$ and $m\in \bZ_{>0}$ there is an associated \emph{valuation ideal}
defined locally by  $\fa_m(v) := \{f \in \cO_X \, \vert \, v(f) \ge m \}$. Note that $\fa_m(v)$ is $\fm_{x}$-primary for each $m \in \bZ_{>0}$. 
For an ideal $\fa\subset \cO_X$ and $v\in \Val_{x,X}$,  we set 
\[
v(\fa) := \min\{v(f)\, \vert\, f\in \fa \cdot \cO_{x,X} \}\in [0, +\infty].\]

\subsection{Normalized volumes of singularities}
 Let $\bk$ be an algebraically closed field of characteristic $0$. For
 an $n$-dimensional klt singularity $x\in (X,D)$ over $\bk$, C. Li \cite{Li15a} introduced the normalized volume function $\hvol_{x,(X,D)}:\Val_{x,X}\to \bR_{>0}\cup\{+\infty\}$.   
 Recall that for $v\in \Val_{x,X}$, 
\[
\nvol_{x,(X,D)}(v) := 
\begin{cases}
 A_{X,D}(v)^n\cdot \vol(v)  & \text{ if } A_{X,D}(v)< +\infty \\
 +\infty & \text{ if } A_{X,D}(v) = +\infty
 \end{cases},\]
where $A_{X,D}(v)$ and $\vol(v)$ denote the \emph{log discrepancy} and \emph{volume} of $v$. 
  As defined in  \cite{ELS03}, the \emph{volume} of $v$ is given by
\[
\vol(v):=  \limsup_{m \to \infty} 
\frac{ \ell(\cO_{x,X} / \fa_m(v) ) }{ m^n /n!}. 
\]
By \cite{ELS03,Mus02,LM09,Cut13},
\[
\vol(v)=   \lim_{m \to \infty} 
\frac{ \hs( \fa_m(v) ) }{ m^n }. \]
The \emph{log discrepancy} of $v$, denoted $A_{X,D}(v)$, is defined in \cite{JM12,BdFFU15} (and \cite{LL16} for the case of klt pairs). 

 The \emph{normalized volume} (also known as \emph{local volume}) of the singularity
 $x\in (X,D)$ is given by \[
 \hvol(x,X,D):=\inf_{v\in\Val_{x,X}}\hvol_{x,(X,D)}(v).
 \]
 When $\bk$ is uncountable, the above infimum is a minimum \cite{Blu16}.

 The following characterization of normalized volumes using log canonical thresholds and multiplicities of ideals is crucial in our study. Note that the right hand side of \eqref{lcte} was studied by de Fernex, Ein and Musta\c{t}\u{a} \cite{dFEM04}
 when $x\in X$ is smooth and $D=0$.
 
 \begin{thm}[{\cite[Theorem 27]{Liu16}}]\label{eqhvol}
 With the above notation, we have
 \begin{equation}\label{lcte}
 \hvol(x, X, D)=\inf_{\fa\colon \fm_{x}\textrm{-primary}}\lct(X,D;\fa)^n\cdot\e(\fa).
 \end{equation}
 
 \end{thm}
 
 The following theorem provides an alternative characterization
 of K-semistability using the affine cone construction. Here we
 state the most general form, and special cases can be found
 in \cite{Li15b, LL16}.
 
\begin{thm}[{\cite[Proposition 4.6]{LX16}}]\label{ksscone}
Let $(Y,E)$ be a log Fano pair of dimension $(n-1)$ over an algebraically closed field $\bk$ of characteristic $0$. For $r\in\bN$ satisfying
$L:=-r(K_Y+E)$ is Cartier, the affine cone $X=C(Y,L)$ is
defined by $X:=\Spec\oplus_{m\geq 0}H^0(Y,L^{\otimes m})$. Let $D$ be the $\bQ$-divisor on $X$ corresponding to $E$. Denote by $x$ the cone vertex of $X$. Then 
\[
\hvol(x,X,D)\leq r^{-1}(-K_Y-E)^{n-1},
\]
and the equality holds if and only if $(Y,E)$ is log K-semistable.
\end{thm}

\subsection{$\bQ$-Gorenstein flat families of klt pairs}
In this section, the field $\bk$ is not assumed to be algebraically closed.

\begin{defn}\hspace{.1 in}

\begin{enumerate}[label=(\alph*)]
\item Given a normal variety $T$, a \emph{$\bQ$-Gorenstein flat family of klt pairs over} $T$ consists of
a surjective flat morphism $\pi:\cX\to T$ from a variety $\cX$, and 
an effective $\bQ$-divisor $\cD$ on $\cX$ avoiding codimension $1$ singular points of $\cX$,
such that the following conditions hold:
\begin{itemize}
    \item All fibers $\cX_t$ are connected, normal and not contained in $\Supp(\cD)$;
    \item $K_{\cX/T}+\cD$ is $\bQ$-Cartier;
    \item $(\cX_t,\cD_t)$ is a klt pair for any $t\in T$.
\end{itemize}
\item A $\bQ$-Gorenstein flat family of klt pairs $\pi:(\cX,\cD)\to T$ together with a section $\sigma:T\to\cX$ is called a \emph{$\bQ$-Gorenstein flat family of klt singularities}. We denote
by $\sigma(\gt)$ the unique closed point of $\cX_{\gt}$ lying over $\sigma(t)\in\cX_t$.
\end{enumerate}
\end{defn}

\begin{prop} Let $\pi:(\cX,\cD)\to T$ be a $\bQ$-Gorenstein flat family of klt pairs over a normal variety $T$. The following hold. 
\begin{enumerate}
    \item There exists a closed subset $\cZ$ of $\cX$ 
    of codimension at least $2$, such that $\cZ_t$ has codimension at least $2$ in $\cX_t$ for every $t\in T$, and $\pi:\cX\setminus \cZ\to T$ is a smooth morphism.
    \item $\cX$ is normal.    
    \item For any morphism $f:T'\to T$ from a normal variety $T'$
    to $T$, the base change $\pi_{T'}:(\cX_{T'},\cD_{T'})=(\cX,\cD)
    \times_{T}T'\to T'$ is a $\bQ$-Gorenstein flat family of klt 
    pairs over $T'$, and $K_{\cX_{T'}/T'}+\cD_{T'}=g^*(K_{\cX/T}+\cD)$
    where $g:\cX_{T'}\to \cX$ is the base change of $f$.
\end{enumerate}
\end{prop}

\begin{proof}
(1) Assume $\pi$ is of relative dimension $n$. Let 
$\cZ:=\{x\in \cX\mid \dim_{\kappa(x)}\Omega_{\cX/T}\otimes\kappa(x)>n\}$.
It is clear that $\cZ$ is Zariski closed. Since $\bk$ is of characteristic $0$,
$\cZ_t=\cZ\cap\cX_t$ is the singular locus of $\cX_t$. Hence $\mathrm{codim}_{\cX_t}\cZ_t\geq 2$
 because $\cX_t$ is normal.

(2) From (1) we know that $\cZ$ is of codimension at least $2$ in 
$\cX$, and $\cX\setminus \cZ$  is smooth over $T$. Thus
$\cX\setminus (\cZ\cup\pi^{-1}(T_{\mathrm{sing}}))$ is regular,
and $\cZ\cup\pi^{-1}(T_{\mathrm{sing}})$ has codimension at least 
$2$ in $\cX$. So $\cX$ satisfies property $(R_1)$. Since $\pi$ is flat,
 for any point $x\in\cX_t$ we have $\depth(\cO_{x,\cX})=\depth(\cO_{x,\cX_t})+\depth(\cO_{t,T})$
 by \cite[(21.C) Corollary 1]{Mat80}. Hence it is easy to
 see that $\cX$ satisfies property ($S_2$) since both $\cX_t$ and $T$
 are normal. Hence $\cX$ is normal.
 
(3) Let $\cZ_{T'}:=\cZ\times_T T'$, and note that $\cX_{T'}\setminus\cZ_{T'}$
is smooth over $T'$. Since the fibers of $\pi_{T'}$ and $T'$ are
irreducible, $\cX_{T'}$ is also irreducible. Thus
the same argument of (2) implies that $\cX_{T'}$ satisfies 
both ($R_1$) and ($S_2$), which means $\cX_{T'}$ is normal. Since
$\pi|_{\cX\setminus\cZ}$ is smooth, we know that $K_{\cX_{T'}/T'}+\cD_{T'}$
and $g^*(K_{\cX/T}+\cD)$ are $\bQ$-linearly equivalent after restriting
to $\cX_{T'}\setminus \cZ_{T'}$. Since $\cZ_{T'}$ is of 
codimension at least $2$ in $\cX_{T'}$, the $\bQ$-linear equivalence over
$\cX_{T'}\setminus\cZ_{T'}$ extends to $\cX_{T'}$. Thus we finish
the proof.
\end{proof}

\begin{defn}
\begin{enumerate}[label=(\alph*)]
 \item Let $Y$ be a normal projective variety. Let $E$
 be an effective $\bQ$-divisor on $Y$. We say that
 $(Y,E)$ is a \emph{log Fano pair} if $(Y,E)$ is a klt pair and 
 $-(K_Y+E)$ is $\bQ$-Cartier and ample. We say $Y$ is a 
 \emph{$\bQ$-Fano variety} if $(Y,0)$ is a log Fano pair.
 \item
 Let $T$ be a normal variety. A $\bQ$-Gorenstein  family
 of klt pairs $\varphi:(\cY,\cE)\to T$ is called a 
 \emph{$\bQ$-Gorenstein flat family of log Fano pairs} if $\varphi$ is
 proper and $-(K_{\cY/T}+\cE)$ is $\varphi$-ample.
\end{enumerate}

\end{defn}

The following proposition states a well known result on the behaviour of the log canonical threshold in families. 
See \cite[Corollary 1.10]{Amb16} for a similar statement. 
The proof is omitted because it follows from arguments similar to those in \cite{Amb16}.

\begin{prop}\label{lctsemicont}
Let $\pi:(\cX,\cD)\to T$ be a $\bQ$-Gorenstein flat family of klt pairs over a normal variety $T$. Let $\fa$ be an ideal sheaf of $\cX$. Then
\begin{enumerate}
\item The function $t\mapsto \lct(\cX_t,\cD_t;\fa_t)$ on $T$ is constructible;
\item If in addition $V(\fa)$ is proper over $T$, then the function $t\mapsto \lct(\cX_t,\cD_t;\fa_t)$ is lower semicontinuous with respect to the Zariski topology on $T$.
\end{enumerate}
\end{prop}

\section{Comparison of normalized volumes and normalized colengths}
\subsection{Normalized colengths of klt singularities}\label{hatl}
\begin{defn} Let  $x\in (X,D)$ be a klt singularity over an algebraically closed field $\bk$ of characteristic $0$.
Denote its local ring by $(R,\fm):=(\cO_{x,X},\fm_x)$. 
\begin{enumerate}[label=(\alph*)]
\item Given constants $c\in\bR_{>0}$ and $k\in\bN$, we define
the \emph{normalized colength of $x\in (X,D)$ with respect
to $c,k$} as 
\[
\widehat{\ell_{c,k}}(x,X,D):=n!\cdot\inf_{\substack{\fm^k\subset\fa\subset\fm\\\ell(R/\fa)\geq ck^n}}\lct(\fa)^n\cdot\ell(R/\fa).
\]
Note that the assumption $\fm^k\subset\fa\subset\fm$ implies $\fa$ is an $\fm$-primary ideal.
\item 
Given a constant $c\in\bR_{>0}$, we define the \emph{asymptotic normalized colength function
of $x\in (X,D)$ with respect to $c$} as 
\[
\widehat{\ell_{c,\infty}}(x,X,D):=\liminf_{k\to\infty}\widehat{\ell_{c,k}}(x,X,D).
\]
\end{enumerate}
\end{defn}
It is clear that $\widehat{\ell_{c,k}}$ is an increasing function in $c$. The main result in this section is the following theorem.

\begin{thm}\label{hvolcolength}
For any klt singularity $x\in (X,D)$ over an algebraically closed field $\bk$ of characteristic $0$, there exists
$c_0=c_0(x,X,D)>0$ such that 
\begin{equation}\label{eq_hvolcolength}
 \widehat{\ell_{c,\infty}}(x,X,D)=\hvol(x, X, D) \quad\textrm{ whenever }0< c\leq c_0.
\end{equation}
\end{thm}

\begin{proof}
We first show the ``$\leq$'' direction.
Let us take a sequence of valuations $\{v_i\}_{i\in\bN}$ such that
$\lim_{i\to\infty}\hvol(v_i)=\hvol(x,X,D)$. We may rescale
$v_i$ so that $v_i(\fm)=1$ for any $i$. Since $\{\hvol(v_i)\}_{i\in\bN}$ are
bounded from above, by \cite[Theorem 1.1]{Li15a} we know that
there exists $C_1>0$ such that $A_{X,D}(v_i)\leq C_1$ for any 
$i\in\bN$. Then by Li's Izumi type inequality \cite[Theorem 3.1]{Li15a},
there exists $C_2>0$ such that $\ord_{\fm}(f)\leq v_i(f)\leq 
C_2\ord_{\fm}(f)$ for any $i\in\bN$ and any $f\in R$.
As a result, we have $\fm^{k}\subset\fa_{k}(v_i)\subset\fm^{\lceil
k/C_2\rceil}$ for any $i,k\in \bN$.
Thus $\ell(R/\fa_k(v_i))\geq\ell(R/\fm^{\lceil k/C_2\rceil})\sim
\frac{\e(\fm)}{n!C_2^{n}}k^n$. Let us take $c_0=\frac{\e(\fm)}{2n!C_2^n}$,
then for $k\gg 1$ we have $\ell(R/\fa_k(v_i))\geq c_0 k^n$
for any $i\in\bN$. Therefore, for any $i\in\bN$ we have
\[
 \widehat{\ell_{c_0,\infty}}(x,X,D)\leq n!\liminf_{k\to\infty}\lct(\fa_k(v_i))^n\ell(R/\fa_k(v_i))
 = \lct(\fa_\bullet(v_i))^n\vol(v_i)\leq \hvol(v_i).
\]
In the last inequality we use $\lct(\fa_\bullet(v_i))\leq A_{X,D}(v_i)$
as in the proof of \cite[Theorem 27]{Liu16}.
Thus $ \widehat{\ell_{c_0,\infty}}(x,X,D)\leq\lim_{i\to\infty}\hvol(v_i)=\hvol(x,X,D)$.
This finishes the proof of the ``$\leq$'' direction.

For the ``$\geq$'' direction, we will show that 
$\widehat{\ell_{c,\infty}}(x,X,D)\geq \hvol(x,X,D)$ for
any $c>0$. By a logarithmic version of the
Izumi type estimate \cite[Theorem 3.1]{Li15a},
there exists a constant $c_1=c_1(x,X,D)>0$ such that 
$v(f)\leq c_1 A_{X,D}(v)\ord_{\fm}(f)$ for any valuation
$v\in \Val_{x,X}$ and any function $f\in R$.
For any $\fm$-primary ideal $\fa$, there exists a
divisorial valuation $v_0\in\Val_{x,X}$ computing $\lct(\fa)$
by \cite[Lemma 26]{Liu16}. Hence we have the following Skoda
type estimate:
\begin{align*}
 \lct(\fa)=\frac{A_{X,D}(v_0)}{v_0(\fa)}\geq \frac{A_{X,D}(v_0)}
 {c_1 A_{X,D}(v_0)\ord_{\fm}(\fa)} =\frac{1}{c_1\ord_{\fm}(\fa)}.
\end{align*}
Let $0<\delta<1$ be a positive number. If
$\fa\not\subset\fm^{\lceil\delta k\rceil}$ and $\ell(R/\fa)\geq c k^n$, then 
\[
\lct(\fa)^n\cdot\ell(R/\fa)\geq \frac{ ck^n}
{c_1^n(\lceil\delta k\rceil-1)^n}\geq \frac{ c}{c_1^n\delta^n}.
\]
If we choose $\delta$ sufficiently small such that
$\delta^n\cdot c_1^n\hvol(x,X,D)\leq n! c$, then for any
$\fm$-primary ideal $\fa$ satisfying $\fm^k\subset\fa\not\subset
\fm^{\lceil\delta k\rceil}$ and $\ell(R/\fa)\geq ck^n$ we have
\[
 n!\cdot \lct(\fa)^n\cdot\ell(R/\fa)\geq \hvol(x,X,D).
\]
Thus it suffices to show
\[
\hvol(x,X,D)\leq n!\cdot\liminf_{k\to\infty}\inf_{\substack{\fm^k\subset\fa\subset\fm^{\lceil\delta k\rceil}\\\ell(R/\fa)\geq ck^n}}\lct^n(\fa)\ell(R/\fa).
\]
By Lemma \ref{colengthmult}, we know that for any $\epsilon>0$
there exists $k_0=k_0(\delta,\epsilon,(R,\fm))$ such that for any $k\geq k_0$ we have
\[
n!\cdot\inf_{\fm^k\subset\fa\subset\fm^{\lceil\delta k\rceil}
}\lct^n(\fa)\ell(R/\fa)\geq (1-\epsilon)\inf_{\fm^k\subset\fa\subset\fm^{\lceil\delta k\rceil}}
\lct(\fa)^n\e(\fa)\geq (1-\epsilon)\hvol(x,X,D).
\]
Hence the proof is finished.
\end{proof}

The following result on comparison between colengths 
and multiplicities is crucial in the proof of Theorem 
\ref{hvolcolength}. Note that Lemma \ref{colengthmult}
is a special case of Lech's inequality \cite[Theorem 3]{Lec60} when $R$ is a regular
local ring.

\begin{lem}\label{colengthmult}
Let $(R,\fm)$ be an $n$-dimensional analytically irreducible
Noetherian local domain. Assume that the residue field
$R/\fm$ is algebraically closed. Then for any positive 
numbers $\delta,\epsilon\in(0,1)$, 
there exists $k_0=k_0(\delta,\epsilon,(R,\fm))$ such that 
for any $k\geq k_0$ and any ideal 
$\fm^{k}\subset\fa\subset\fm^{\lceil\delta k\rceil}$, we have
\[
n!\cdot\ell(R/\fa)\geq (1-\epsilon)\e(\fa).
\]
\end{lem}

\begin{proof}
By \cite[7.8]{KK14} and \cite[Section 4]{Cut13}, $R$ admits a 
\emph{good} valuation $\nu:R\to\bZ^n$ for some total order on $\bZ^{n}$.
 Let $\cS:=\nu(R\setminus\{0\})\subset\bN^n$ and $C(\cS)$ be the closed convex hull of $\cS$. Then we know that
\begin{itemize}
\item $C(\cS)$ is a strongly convex cone;
\item There exists a linear functional $\xi:\bR^n\to\bR$ such that $C(\cS)\setminus\{0\}\subset\xi_{>0}$;
\item There exists $r_0\geq 1$ such that for any $f\in R\setminus\{0\}$, we have
\begin{equation}\label{eq_izumi}
\ord_{\fm}(f)\leq \xi(\nu(f))\leq r_0\ord_{\fm}(f).
\end{equation}
\end{itemize}

Suppose $\fa$ is an ideal satisfying $\fm^k\subset\fa\subset\fm^{\lceil\delta k\rceil}$. Then we have $\nu(\fm^k)\subset\nu(\fa)\subset\nu(\fm^{\lceil\delta k\rceil})$. By \eqref{eq_izumi}, we know that 
\[
\cS\cap\xi_{\geq r_0 k}\subset 
\nu(\fa)\subset\cS\cap\xi_{\geq \delta k}.
\]
Similarly, we have $\cS\cap\xi_{\geq r_0 ik}\subset 
\nu(\fa^i)\subset\cS\cap\xi_{\geq \delta ik}$ for any positive integer $i$.

Let us define a semigroup $\Gamma\subset\bN^{n+1}$ as follows:
\[
\Gamma:=\{(\alpha,m)\in\bN^n\times\bN\colon x\in\cS\cap\xi_{\leq 2r_0 m}\}.
\]
For any $m\in\bN$, denote by $\Gamma_m:=\{\alpha\in\bN^n\colon(\alpha,m)\in\Gamma\}$. It is easy to see $\Gamma$ satisfies \cite[(2.3-5)]{LM09}, thus \cite[Proposition 2.1]{LM09} implies
\[
\lim_{m\to\infty}\frac{\#\Gamma_m}{m^n}=\vol(\Delta),
\]
where $\Delta:=\Delta(\Gamma)$ is a convex body in $\bR^{n}$ 
defined in \cite[Section 2.1]{LM09}. It is easy to see that $\Delta= C(\cS)\cap\xi_{\leq 2r_0}$.

Let us define $\Gamma^{(k)}:=\{(\alpha,i)\in\bN^n\times\bN\colon(\alpha,ik)\in\Gamma\}$. Then we know that $\Delta^{(k)}:=\Delta(\Gamma^{(k)})=k\Delta$. For an ideal $\fa$ and $k\in\bN$ satisfying $\fm^k\subset\fa\subset\fm^{\lceil\delta k\rceil}$, we define 
\[
\Gamma_{\fa}^{(k)}:=\{(\alpha,i)\in\Gamma^{(k)}\colon\alpha\in\nu(\fa^i)\}.
\]
Then it is clear that $\Gamma_{\fa}^{(k)}$ also satisfies 
\cite[(2.3-5)]{LM09}. Since  $\nu(\fa^i)=(\cS\cap\xi_{> 2r_0ik})
\cup\Gamma_{\fa,i}^{(k)}$ and $R/\fm$ is algebraically 
closed, we have $\ell(R/\fa^i)=\#(\Gamma_i^{(k)}\setminus\Gamma_{\fa,i}^{(k)})$
because $\nu$ has one-dimensional leaves. Again by \cite[Proposition 2.11]{LM09}, we have
\[
n!\e(\fa)=\lim_{i\to\infty}\frac{\ell(R/\fa^i)}{i^n}=\lim_{i\to\infty}\frac{\#(\Gamma_i^{(k)}\setminus\Gamma_{\fa,i}^{(k)})}{i^n}
=\vol(\Delta^{(k)})-\vol(\Delta_{\fa}^{(k)}),
\]
where $\Delta_{\fa}^{(k)}:=\Delta(\Gamma_{\fa}^{(k)})$. Since $\Gamma_{\fa,i}^{(k)}\subset\nu(\fa^i)\subset\xi_{\geq\delta ik}$, we know that $\Delta_{\fa}^{(k)}\subset\xi_{\geq\delta k}$. Denote by $\Delta':=C(\cS)\cap\xi_{< \delta}$, then it is clear that $\Delta_{\fa}^{(k)}\subset k(\Delta\setminus\Delta')$.

On the other hand, 
\[
\ell(R/\fa)=\#(\Gamma_{1}^{(k)}\setminus\Gamma_{\fa,1}^{(k)})
\geq\#\Gamma_{k}-\#(\Delta_{\fa}^{(k)}\cap\bZ^n).
\]
Denote by $\Delta_{\fa,k}:=\frac{1}{k}\Delta_{\fa}^{(k)}$,
then $\Delta_{\fa,k}\subset\Delta\setminus\Delta'$. Since $\vol(\Delta_{\fa,k})\leq \vol(\Delta)-\vol(\Delta')$, there exists positive numbers $\epsilon_1,\epsilon_2$ depending only on $\Delta$ and $\Delta'$ such that 
\begin{equation}\label{ineq1}
\vol(\Delta_{\fa,k})\leq\vol(\Delta)-\vol(\Delta')\leq\left(1-\frac{\epsilon_1}{\epsilon}\right)\vol(\Delta) -\frac{\epsilon_2}{\epsilon}.
\end{equation}
Let us pick $k_0$ such that for any $k\geq k_0$ and any $\fm^k\subset\fa\subset\fm^{\lceil\delta k\rceil}$, we have
\[
\frac{\#\Gamma_k}{k^n}\geq (1-\epsilon_1)\vol(\Delta),\qquad
\frac{\#(\Delta_{\fa}^{(k)}\cap\bZ^n)}{k^n}\leq\vol(\Delta_{\fa,k})+\epsilon_2.
\]
Here the second inequality is guaranteed by applying Proposition \ref{convexgeo} to $\Delta_{\fa,k}$
as a sub convex body of a fixed convex body $\Delta$. Thus
\begin{align*}
    \frac{\ell(R/\fa)-(1-\epsilon)n!\e(\fa)}{k^n}&
    \geq \frac{\#\Gamma_k}{k^n}-\frac{\#(\Delta_{\fa}^{(k)}\cap\bZ^n)}{k^n}-(1-\epsilon)(\vol(\Delta)-\vol(\Delta_{\fa,k}))\\
   & \geq (1-\epsilon_1)\vol(\Delta)-\vol(\Delta_{\fa,k})-\epsilon_2-(1-\epsilon)(\vol(\Delta)-\vol(\Delta_{\fa,k}))\\
   & =(\epsilon-\epsilon_1)\vol(\Delta)-\epsilon(\Delta_{\fa,k})-\epsilon_2\\
   &\geq 0.
\end{align*}
Here the last inequality follows from \eqref{ineq1}. Hence we finish the proof.
\end{proof}

\subsection{Normalized volumes under field extensions}\label{sec_fieldext}
In the rest of this section, we use Hilbert schemes to describe
normalized volumes of singularities after a field extension $\bK/\bk$.
Let $(X,D)$ be a klt pair over $\bk$ and $x\in X$ be a 
$\bk$-rational point. Let $Z_k:=\Spec(\cO_{x,X}/\fm_{x,X}^k)$ denote the $k$-th thickening of $x$.  Consider the Hilbert scheme
$H_{k,d}:=\mathrm{Hilb}_{d}(Z_k/\bk)$. For any field extension
$\bK/\bk$ we know that $H_{k,d}(\bK)$ parametrizes ideal sheaves
$\fc$ of $X_\bK$ satisfying $\fc\supset\fm_{x_{\bK},X_{\bK}}^k$
and $\ell(\cO_{x_{\bK},X_{\bK}}/\fc)=d$. In particular,
any scheme-theoretic point $h\in H_{k,d}$ corresponds to an ideal $\fb$ of $\cO_{x_{\kappa(h)},X_{\kappa(h)}}$ satisfying those two conditions, and we denote by $h=[\fb]$.

\begin{prop}\label{fieldext}
Let $\bk$ be a field of characteristic $0$. Let $(X,D)$ be a klt pair over $\bk$. Let $x\in X$ be a $\bk$-rational point. Then 
\begin{enumerate}
    \item For any field extension $\bK/\bk$ with $\bK$ algebraically closed, we have 
    \[
    \widehat{\ell_{c,k}}(x_{\bK}, X_{\bK},D_{\bK})=
    n!\cdot\inf_{d\geq ck^n, ~ [\fb]\in H_{k,d}}
    d\cdot\lct(X_{\kappa([\fb])}, D_{\kappa([\fb])};\fb)^n.
    \]
    \item With the assumption of (1), we have $$\hvol(x_{\bK}, X_{\bK},D_{\bK})=\hvol(x_{\bar{\bk}}, X_{\bar{\bk}},D_{\bar{\bk}}).$$
\end{enumerate}
\end{prop}

\begin{proof} (1) We first prove the ``$\geq$'' direction.
By definition, $\widehat{\ell}_{c,k}(x_{\bK}, X_{\bK}, D_{\bK})$ is the 
infimum of $n!\cdot\lct(X_{\bK},D_{\bK};\fc)^n\ell(\cO_{X_{\bK}}/\fc)$ where
$\fc$ is an ideal on $X_{\bK}$ satisfying $\fm_{x_{\bK}}^k\subset
\fc\subset\fm_{x_{\bK}}$ and $\ell(\cO_{X_{\bK}}/\fc)=:d\geq ck^n$. Hence $[\fc]$ represents a point in 
$H_{k,d}(\bK)$. Suppose $[\fc]$ is lying over a scheme-theoretic point $[\fb]\in H_{k,d}$, then it is clear that $(X_{\bK},D_{\bK},\fc)\cong (X_{\kappa([\fb])},D_{\kappa([\fb])},\fb)\times_{\Spec(\kappa([\fb]))}\Spec(\bK)$. Hence $\lct(X_{\bK},D_{\bK};\fc)=\lct(X_{\kappa([\fb])},D_{\kappa([\fb])};\fb)$ by \cite[Proposition 7.13]{JM12}, and the ``$\geq$'' direction is proved.

Next we prove the ``$\leq$'' direction. By Proposition \ref{lctsemicont}, we know that the function $[\fb]\mapsto\lct(X_{\kappa([\fb])},D_{\kappa([\fb])}; \fb)$ on $H_{k,d}$ is constructible and lower semicontinuous.
Denote by $H_{k,d}^{\mathrm{cl}}$ the set of closed points in $H_{k,d}$. Since the set of closed points are dense in any stratum of $H_{k,d}$ with respect to the $\lct$ function, we have the following equality:
\[
n!\cdot\inf_{d\geq ck^n, ~ [\fb]\in H_{k,d}}
    d\cdot\lct(X_{\kappa([\fb])}, D_{\kappa([\fb])};\fb)^n
    =n!\cdot\inf_{d\geq ck^n, ~ [\fb]\in H_{k,d}^{\mathrm{cl}}}
    d\cdot\lct(X_{\kappa([\fb])}, D_{\kappa([\fb])};\fb)^n
\]
Any $[\fb]\in H_{k,d}^{\mathrm{cl}}$ satisfies that $\kappa([\fb])$ is an algebraic extension of $\bk$. Since $\bK$ is algebraically closed, $\kappa([\fb])$ can be embedded into $\bK$ as a subfield. Hence there exists a point $[\fc]\in H_{k,d}(\bK)$ lying over $[\fb]$. Thus similar arguments implies that $\lct(X_{\bK},D_{\bK};\fc)=\lct(X_{\kappa([\fb])},D_{\kappa([\fb])};\fb)$, and the ``$\leq$'' direction is proved.
\medskip

(2) From (1) we know that $\widehat{\ell_{c,k}}(x_{\bK},X_{\bK},D_{\bK})=\widehat{\ell_{c,k}}(x_{\bar{\bk}},X_{\bar{\bk}},D_{\bar{\bk}})$ for any $c,k$. Hence it follows from Theorem \ref{hvolcolength}.
\end{proof}

The following corollary is well-known to experts.
We present a proof here using normalized volumes.

\begin{cor}
 Let $(Y,E)$ be a log Fano pair over
 a field $\bk$ of characteristic 
 $0$. The following are equivalent:
 \begin{enumerate}[label=(\roman*)]
  \item $(Y_{\bar{\bk}},E_{\bar{\bk}})$ is log K-semistable;
  \item $(Y_{\bK},E_{\bK})$ is log K-semistable
  for some field extension $\bK/\bk$ with $\bK=\overline{\bK}$;
  \item $(Y_{\bK},E_{\bK})$ is log K-semistable
  for any field extension $\bK/\bk$ with $\bK=\overline{\bK}$.
 \end{enumerate}
 We say that $(Y,E)$ is \emph{geometrically log K-semistable} if 
 one (or all) of these conditions holds.
\end{cor}

\begin{proof}
Let us take the affine cone $X=C(Y,L)$ with $L=-r(K_Y+E)$  Cartier. Let $D$ be the $\bQ$-divisor on $X$ corresponding to $E$. Denote by$x\in X$ the cone vertex of $X$. Let $\bK/\bk$ be a field extension with $\bK=\overline{\bK}$. 
Then  Theorem \ref{ksscone} implies that $(Y_{\bK},E_{\bK})$ is log K-semistable if and only if $\hvol(x_{\bK}, X_{\bK},D_{\bK})=r^{-1}(-K_Y-E)^{n-1}$. Hence the corollary is a consequence of Proposition \ref{fieldext} (2).
\end{proof}

We finish this section with a natural speculation. Suppose 
$x\in (X,D)$ is a klt singularity over a field $\bk$ of characteristic
zero that is not necessarily algebraically closed. The definition of
normalized volume of singularities extend verbatimly to
$x\in (X,D)$ which we also denote by $\hvol(x,X,D)$. Then
we expect $\hvol(x,X,D)=\hvol(x_{\bar{\bk}},X_{\bar{\bk}}, D_{\bar{\bk}})$,
i.e. normalized volumes are stable under base change to algebraic
closures. Such a speculation should be a consequence of the \emph{Stable
Degeneration Conjecture (SDC)} stated in \cite[Conjecture 7.1]{Li15a}
and \cite[Conjecture 1.2]{LX17} which roughly says that a $\hvol$-minimizing
valuation $v_{\min}$ over $x_{\bar{\bk}}\in (X_{\bar{\bk}}, D_{\bar{\bk}})$ is unique and quasi-monomial,
so $v_{\min}$ is invariant under the action of $\mathrm{Gal}(\bar{\bk}/\bk)$
and hence has the same normalized volume as its restriction
to $x\in (X,D)$.

\section{Uniform approximation of volumes by colengths}\label{sec:unifapp}

In this section, we prove the following result that gives an approximation of the volume of valuation by the colengths of its valuation ideals. 
The 
result  is a consequence of arguments in \cite[Section 3.4]{Blu16} (which in turn relies on ideas in \cite{ELS03})
and properties of the Hilbert--Samuel function.

\begin{thm}\label{thm:volcol}
Let $ \pi : (\cX, \cD) \to T$  together with a section $\sigma:T \to \cX$ be a $\bQ$-Gorenstein flat family of klt singularities. Set $n=\dim(\cX)-\dim(T)$.
For every $A\in \bR_{>0}$ and $\epsilon>0$, 
there exists a positive integer $N$ so that the following holds: 
If $t\in T$ and $v\in \Val_{\sgt, \cXgt}$ satisfies $v(\fm_{\sgt} )=1$ and $A_{\cXgt,\cDgt}(v)\leq A$, 
then
\[
\frac{ \ell ( \cO_{\sgt,\cXgt }/ \fa_{ m}(v))}{m^n/ n!} 
\leq 
\vol(v) + \epsilon 
\]
for all positive integers $m$ divisible by $N$.
\end{thm}

 We begin by approximating the volume of a valuation by the multiplicity of its valuation ideals. 
\begin{prop}\label{prop:ELS}
Let $x\in (X,D)$ be a klt singularity defined over an algebraically closed field $\bk$ and $r$ a positive integer such that $r(K_{X}+D)$ is Cartier.
Fix $v\in \Val_{x,X}$ satisfying $v(\fm_x)=1$ and $A_{X,D}(v)<+\infty$. 
\begin{itemize}
\item[(a)] If $x\in X_{\sing} \cup \Supp(D)$, then for all $m \in \bZ_{>0}$ we have
\begin{equation*}
\frac{
\hs(\fa_m(v))^{1/n}}{m} 
\leq 
\vol(v)^{1/n} + 
\frac{ \lceil{A_{X,D}(v)\rceil} \hs(\fm_x)^{1/n} }{m} + 
\frac{ \hs\left(  \cO_{X}(-rD)\cdot  \Jac_X+\fm_x^{m} \right)^{1/n}}{m}.
\end{equation*}

\item[(b)] If $x\notin  X_{\sing} \cup \Supp(D)$, then for all $m \in \bZ_{>0}$ we have
\begin{equation*}
\frac{
\hs(\fa_m(v))^{1/n}}{m} 
\leq 
\vol(v)^{1/n} + 
\frac{ \lceil{A_{X,D}(v)\rceil} \hs(\fm_x)^{1/n} }{m}.
\end{equation*}
\end{itemize}
\end{prop}

\begin{proof}
Fix $v\in \Val_{x,X}$ satisfying $v(\fm_x)=1$ and $A_{X,D}(v) < +\infty$. 
To simplify notation, we set $\ab:= \ab(v)$ and $A:= \lceil A_{X,D}(v) \rceil$. 
By \cite[Theorem 7.2]{Blu16}, 
\begin{equation}\label{eq:ELS3}
(\Jac_X \cdot \cO_{X}(-rD) )^\ell \fa_{(m+A) \ell} \subset   \left(\fa_{m} \right)^\ell
\end{equation}
for all $m,\ell \in \bZ_{>0}$. 
Since $v(\fm_x)=1$, we see
$\fm_x^m \subset \fa_m$ for all $m\in \bZ_{>0}$. 
As in the proof of \cite[Proposition 3.7]{Blu16}, 
it follows from the previous inclusion combined with  \eqref{eq:ELS3} that
\begin{equation}
(\Jac_X \cdot \cO_X(-rD) +\fm_x^m)^\ell  \fa_{(m+A)\ell} \subset   \overline{{(\fa_{m})}^\ell}.
\end{equation}
for all $m \in \bZ_{>0}$. 
We now apply Teissier's Minkowski inequality \cite[Example 1.6.9]{LazPAG} to the previous inclusion and find that
\begin{equation}\label{eq:Min}
\ell \hs(  \fa_m  )^{1/n} 
\leq
  \ell \cdot  \hs(  \Jac_X \cdot \cO_{X}(-rD)+\fm_x^{m} )^{1/n} + 
 \hs( \fa_{(m+A)\ell})^{1/n}.
\end{equation}
Dividing both sides of \eqref{eq:Min} by $m \cdot \ell$ and taking the limit as $\ell \to \infty$ gives
\[
\frac{\hs(\fa_m)^{1/n}}{m} 
\leq 
\frac{\hs( \Jac_X \cdot \cO_X(-rD)+\fm_x^{m} )^{1/n} }{m} + 
 \left( \frac{m+A}{m} \right)\vol(v)^{1/n}.
\]
Since $\fm_x^m\subset \fa_m$ for all $m\in \bZ_{>0}$, $\vol(v) \leq \hs(\fm_x)$ and the desired inequality follows. In the case when $x \notin X_{\sing} \cup \Supp(D)$, the stronger inequality follows from a similar argument and the observation that $(\Jac_X \cdot \cO_{X}(-rD) )$ is trivial in a neighborhood of $x$.
\end{proof}

 Before proceeding, we recall the following defintion of the Jacobian ideal. If $X$ is a variety of dimension $n$, then the \emph{Jacobian ideal} of $X$, denoted $\Jac_X$, the $n$-fitting ideal of $\Omega_X$. More generally, if $\pi:\cX \to T$ is flat morphism of varieties and $n=\dim(\cX)-\dim(T)$, then 
 the \emph{Jacobian ideal} of $\pi$, denoted $\Jac_{\cX/T}$, is $n$-th fitting ideal of $\Omega_{\cX/T}$.

\begin{prop}\label{prop:jacbounded}
With the same assumptions  as in Theorem \ref{thm:volcol}, fix a positive integer $r$ such that $r(K_{\cX/T}+\cD)$ is Cartier. Then, for every $\epsilon>0$, there exists $M$ so that the following holds: If $t\in T$ satisfies  $\st \in V(\Jac_{ \cXt}) \cup \Supp(\cDt)$, then
\[
\frac{\hs\left(\Jac_{\cXt}\cdot \cO_{\cXt}(-r \cDt) + \fm_{\st}^m \right) }{m^n} \leq \epsilon .
\]
for all $m \geq M$.
\end{prop}

\begin{proof}
To simplify notation, set $Z= \{ t\in T \, \vert \,  \st  \in V(\Jac_{ \cXt}) \cup \Supp(\cDt) \}$. 
 We will prove the following claim: for each $\epsilon>0$, there exists a nonempty open set $U\subset T$ and a positive integer $M$ such that if $t\in U \cap Z$, then
 \[
\frac{\hs(\Jac_{\cXt}\cdot \cO_{\cXt}(-r \cDt) \ + \fm_{\st}^m) }{m^n} \leq \epsilon\]
 for all $m \geq M$. 
By inducting on the dimension of $T$, the result will follow.

  We proceed to prove the claim. It is enough to consider the case when $\cX$ and $T$ are affine, since we may replace $\pi$ with its restiction to a nonempty open subset of $T$ and $\cX$ with an open subset containing $\sigma(T)$. 
 Next, 
note that   $\Jac_{\cXt} =\Jac_{\cX/T} \cdot \cO_{\cX}$ for each $t\in T$,  since the formation of fitting ideals commute with base change \cite[Tag 0C3D]{SPA}. Hence, $Z = \sigma^{-1}(V(\Jac_{\cX/T} ) \cup \Supp(\cD) )$ and is closed in $T$. 
  Now, if $T\setminus Z \neq \emptyset$, then the above claim (trivially) holds with  $U= T\setminus Z$. Therefore, we consider the case when $Z=T$. 

Choose
 a nonempty affine open set $U\subset T$  and  $g\in\Jac_{\cX/T} \cdot \cO_{X}(-r\cD)( \pi^{-1}(U))$ such that  the restriction of $g$ to $\cO_{\st,\cXt}$, denoted $g_t$, is nonzero for all $ t \in U$. Set $R_t := \cO_{\st,\cXt}/(g_t)$  and $\widetilde{\fm}_t = \fm_{\st} \cdot  R_t$ for each $t\in U$. Now,
\[
\hs( \Jac_{\cXt} \cdot \cO_{\cXt}(-r\cDt) + \fm_{\sigma(t)}^m) \leq \hs((g_t)+ \fm_{\sigma(t)}^m)  \leq  n!  \cdot \ell (\cO_{\st,\cXt} /((g_t)+ {\fm_{\st}}^m)) \cdot \hs( \fm_{\st} ),\]
where the first inequality follows from the inclusion $ \Jac_{\cXt} \cdot \cO_{\cXt}(-r\cDt)+ \fm_{\sigma(t)}^m \subset 
(g_t)+ \fm_{\sigma(t)}^m$ and the second is precisely Lech's inequality \cite[Theorem 3]{Lec60}.
Thus, 
\begin{equation}\label{eq:jacbounded}
\frac{\hs( \Jac_{\cXt}\cdot \cO_{\cXt}(-r\cDt) + \fm_{\st}^m)}{m^n}  \leq 
n!\left( \frac{  \ell( R_t/  \widetilde{\fm}_{t}^m  )}{m^n} \right)  \cdot  \hs( \fm_{\st} )  .\end{equation}
for all $t\in U$.
By Proposition \ref{prop:hs}, we may shrink $U$  so that $U\ni t \mapsto \hs( \fm_{\st})$ and 
 $U\ni t \mapsto \ell( R_t/\widetilde{\fm}_{t}^m))$, for all $m\geq 1$, are constant.  Since $\dim R_t=n-1$, 
we have $  \ell( R_t/\widetilde{\fm}_{t}^m)= O(m^{n-1})$. Therefore, there exists an integer $M$ so that 
\[
n!\left( \frac{  \ell( R_t/ {\widetilde{\fm}_{t}}^m)}{m^n} \right)  \cdot  \hs(  \fm_{\st} ) \leq \epsilon
\]
for all $m \geq M$ and $t\in U$. This completes the claim.\end{proof}

The following proposition is a consequence of results in Appendix \ref{app:HS}.

\begin{prop}\label{prop:hsbounded}
Keep the assumptions and notation in Theorem \ref{thm:volcol}, and fix an integer $k \in \bZ_{>0}$. 
Then, for any $\epsilon>0$, 
there exists $M \in \bZ_{>0}$ so that the following holds: 
For any point $t\in T$ and ideal $\fa \subset \cO_{\sgt ,\cXgt}$ satisfying $ \fm_{\sgt}^k\subset \fa \subset \fm_{\sgt}$,
\[
\frac{
\ell ( \cO_{\sgt,\cXgt}/ \fa^{m})}{m^n/ n!} \leq \hs(\fa) + \epsilon
\]
for all $m \geq M$. 
\end{prop}

\begin{proof}
Set $d := \max\{ \ell(\cO_{\sgt,\cXgt }/\fm_{\sgt}^k ) \, \vert \, t\in T \}$, 
and consider the union of Hilbert schemes $\cH := \bigcup_{m=1}^{d} \Hilb_m({\cZ}_k/T)$,
where ${\cZ}_k= \Spec( \cO_{\cX}/ \cI_{\sigma(T)}^k)$. Let
$\tau$ denote the morphism $\cH \to T$. A point $h \in H$ 
corresponds to the ideal $\fb_h = \fb \cdot \cO_{\cX\times_T \kappa(h)}$, where $\fb$ is the universal ideal sheaf on $\cX \times_T \cH$. 
By applying Proposition \ref{prop:hs} to the irreducible components of $\cH$ endowed with reduced scheme structure, we see that the set of functions $\{ \HS_{\fb_{h}} \, \vert \, h \in H \}$  is finite. 

Next, fix $\epsilon>0$. By the previous paragraph, there exists $M\in \bZ_{>0}$ so that 
\begin{equation}\label{eq:hsbounded}
\frac{ \HS_{\fb_h}(m)
)}{m^n/ n!} \leq \hs(\fb_h) + \epsilon
\end{equation}
for all $m \geq M$. 
Now, consider a point $t\in T$  and an ideal $\fa \subset \cO_{\cXgt}$  satisfying 
$\fm_{\sgt}^k \subset \fa \subset \fm_{\sgt}$. 
Since 
\[
\ell( \cO_{\sgt,\cXgt}/ \fa)
\leq
 \ell( \cO_{\sgt,\cXgt}/ \fm_{\sigma(\overline{t})} ^{k} ) 
=
 \ell( \cO_{\st, \cXt }/ \fm_{\sigma(t)}^{k} ) \leq d,\]
 there is a map  $\rho : \Spec (\overline{\kappa(t)}) \to \cH$  such that $\fa =\fb_{\rho(0) } \cdot \cO_{\cX \times_T \overline{\kappa(t)}}$. Therefore, $\HS_{\fa}= \HS_{\fb_{\rho(0)}}$ and \eqref{eq:hsbounded} implies
 \[
\frac{ \ell ( \cO_{\sgt, \cXgt }/ \fa^{m})}{m^n/ n!} \leq \hs(\fa) + \epsilon
 \]
 for all $m\geq M$. 
\end{proof}

We will now deduce Theorem \ref{thm:volcol} from Propositions \ref{prop:ELS}, \ref{prop:jacbounded}, and \ref{prop:hsbounded}.

\begin{proof}[Proof of Theorem \ref{thm:volcol}]
To simplify notation, we set
\[
W_{t} =\{ v\in \Val_{\sgt,\cXgt} \,\vert \,   v(\fm_{\sgt})=1\, \text{ and } A_{\cXgt,\cDgt}(v) \leq A\}\]
for each $t\in T$. In order to prove the theorem, it suffices to prove the following claim: for every $\epsilon>0$, there exists an integer $N$ so that if $t\in T$, then 
\[
\left(
\frac{ \ell ( \cO_{\sgt,\cXgt}/ \fa_{ m}(v))}{(m)^n/ n!} \right)^{1/n} \leq \vol(v)^{1/n} + \epsilon 
\]
for all  $v\in W_t$ and  $m\in \bZ_{>0}$ divisible by $N$. Indeed, if $v\in W_t$, then $\vol(v) \leq \hs(\fm_{\sgt})$. Since the set $\{\hs(\fm_{\sgt}) \, \vert \, t\in T\}$ is bounded from above by Proposition \ref{prop:hs}, the claim implies the conclusion of the theorem.

We now fix $\epsilon>0$ and proceed to bound the latter two terms in Proposition \ref{prop:ELS}.1.
First, we apply Proposition  \ref{prop:hsbounded} to find a positive integer $M_1$ so that 
\[
 \frac{ A\cdot  \hs(\fm_{\sgt})^{1/n} }{M_1}  \leq \epsilon /4
\]
for all $t\in T$. Next, we apply Proposition \ref{prop:jacbounded} to find a positive integer $M_2$ so that the following holds: if $t\in T$ and $\sgt \in   V(\Jac_{\cXgt}) \cup \Supp(\cDgt)$, then
\[ 
 \frac{ \hs( \Jac_{\cXgt} \cdot \cO_{\cXgt}(-r\cDgt)+{\fm_{\sgt}}^{m'})^{1/n}}{m'} 
=
\frac{ \hs( \Jac_{\cXt} \cdot \cO_{\cXt}(-r\cDt)+{\fm_{\st}}^{m'})^{1/n}}{m'} 
< \epsilon/4.\]
for all $m' \geq M_2$. Now, set $m' := \max \{M_1,M_2\}$. Proposition \ref{prop:ELS}  implies that if $t\in T$, then 
\begin{equation}\label{eq:volcol1}
\frac{\hs(\fa_{m'}(v))^{1/n}}{m'} \leq \vol(v)^{1/n}+ \epsilon/2
\end{equation}
for all $v\in W_t$.

Next, note that if $t\in T$ and $v\in W_t$, then $ \fm_{\sgt}^{m'} \subset \fa_{m'}(v)$. Therefore, we may  apply Proposition \ref{prop:hsbounded}  to find an integer $M$ such that if $t\in T$ and $v\in W_t$, then
 \begin{equation}\label{eq:volcol2}
\left(
\frac{
\ell ( \cO_{\sgt,\cXgt }/ (\fa_{m'}(v))^\ell )}{( m' \cdot   \ell )^n/ n!}  \right)^{1/n} \leq \frac{\hs(\fa_{m'}(v))^{1/n} }{m'} + \epsilon/2
\end{equation}
for all $\ell \geq M$. Thus, if $t\in T$ and $v\in W_t$, then
\[
\left(
\frac{\ell ( \cO_{\sgt ,\cXgt }/ (\fa_{m'  \cdot \ell }(v)) )}{(m'  \cdot \ell )^n/ n!}  
\right)^{1/n}   
 \leq 
\frac{\hs(\fa_{m'}(v))^{1/n}}{m'} + \epsilon/2 \leq 
 \vol(v) + \epsilon  \]
 for all $\ell \geq M$,
 where the first inequality follows from \eqref{eq:volcol2} combined with the inclusion $\fa_{m'}(v)^\ell \subset \fa_{m'  \cdot  \ell}(v)$ and the second inequality  from \eqref{eq:volcol1}. Therefore, setting $N:=m' \cdot M$ completes the claim. 
\end{proof}

\section{Li's Izumi and properness estimates in families}\label{sec:izumi}

In this section, we generalize results of \cite{Li15a} to families of klt singularities. These results will be used to prove Theorem \ref{thm:convncol}. 

\begin{thm}[Izumi-type Estimate]\label{thm:izumifamily}
Let $\pi: (\cX,\cD) \to T$ together with a section $\sigma : T \to \cX$ be a $\bQ$-Gorenstein flat family of klt singularities over a variety $T$. There exists a constant $K_0>0$ so that the following holds: If $t\in T$ and $v\in \Val_{\sgt,\cXgt}$  satisfies $A_{\cXgt,\cDgt}(v)<+\infty$, then
\[ 
v(g) \leq  K_0 \cdot A_{\cXgt,\cDgt}(v)  \cdot \ord_{\sgt}(g)
\]
for all $g\in \cO_{\sgt,\cXgt}$. 
\end{thm}

\begin{thm}[Properness Estimate]\label{thm:properness}
Let $\pi: (\cX,\cD) \to T$ together with a section $\sigma : T \to \cX$ be a $\bQ$-Gorenstein flat family of klt singularities over a variety $T$. There exists a constant $K_1>0$ so that the following holds: 
If $t\in T$ and $v\in \Val_{\sgt,\cXgt}$ satisfies $A_{\cXgt,\cDgt}(v)<+\infty$, then 
\[
  \frac{K_1 \cdot A_{\cXgt,\cDgt}(v)}{v(\fm_{\sgt})} \leq A_{\cXgt,\cDgt}(v)^n  \cdot \vol(v).
\]
\end{thm}

The proofs of these theorems rely primarily on the result and techniques found in \cite{Li15a}. The main new ingredient can be found in Proposition \ref{prop:classicalizumi}, which is proved using arguments of \cite{BFJ14} and \cite[Appendix II]{Li15a}.

\subsection{Order functions}

Let $X$ be a normal variety defined over an algebraically closed field $\bk$ and $x \in X$ a closed point. 
For $g\in \cO_{x,X}$ the \emph{order of vanishing} of $g$ at $x$ is defined as
\[
\ord_x(g) := \max \{ j \geq 0 \, \vert \, g\in \fm_x^j \} .\]
If $X$ is smooth at $x$, then $\ord_x$ is a valuation of the function field of $X$. In the singular case, $\ord_x$ may fail to be a valuation. For example, the inequality 
\[
\ord_{x}(g^{n+n'}) \geq \ord_x(g^{n}) + \ord_x(g^{n'})\] may be strict. Following \cite{BFJ14}, we consider an alternative function $\widehat{\ord}_x$, which is defined by
\[
\widehat{\ord}_x(g) := \lim_{n \to \infty} \frac{1}{n} \ord_x(g^n) = \sup_{n} \frac{1}{n} \ord_x(g^n)
.\] 

Let $\nu : X^+ \to X$ denote the normalized blowup of $\fm_x$
and write
\[
\fm_{x} \cdot \cO_{X^+} = \cO_{X}\left( - \sum_{i=1}^{r} a_i E_i \right),
\]
where the $E_i$ are prime divisors on $X$ and each $a_i \in \bZ_{>0}$. The following statement, which was proved in \cite[Theorem 4.3]{BFJ14}, gives an interpretation of  $\widehat{\ord}_x$ in terms of the exceptional divisors of $\nu$. 

\begin{prop}\label{prop:ordhat}
For any function $g\in \cO_{x,X}$ and $m \in \bZ_{>0}$, 
\begin{enumerate}
\item $
\widehat{\ord}_x(g)  = \min_{i=1,\ldots,r} \frac{ \ord_{E_i}(g)}{a_i}$ and 
\item 
$\widehat{\ord}_x(g) \geq m$ if and only if $g \in \overline{ \fm_x^m}$.
\end{enumerate}
\end{prop}

Building upon results in \cite[Section 4.1]{BFJ14}, we show a comparison between $\ord_x$ and $\widehat{\ord}_x$. 

\begin{prop}\label{prop:ordcomp}
If there exists a $\bQ$-divisor $D$ such that $(X,D)$ is klt pair, then 
\[
\ord_x(g) \leq \widehat{\ord}_x(g)   \leq (n+1) \ord_x(g) 
\]
for all $g\in \cO_{x,X}$ and $n= \dim(X)$. 
\end{prop}

\begin{proof}
The first inequality follows from the definition of $\widehat{\ord}_x(g)$ as a supremum. 
For the second inequality, assume $m := \ord_x(g) >0$ and note that $g\notin \fm_x^{m+1}$. Since
\[  
\overline{ \fm_x^{m+n}} \subset \cJ( (X,D), \fm_x^{m+n} ) 
 = \fm_x^{m+1}  \cdot \cJ( (X,D), \fm_x^{n-1}) \subset {\fm_x}^{m+1},
\]
where the first inclusion follows from the fact that $(X,D)$ is klt and the second from Skoda's Theorem  \cite[9.6.39]{LazPAG}, we see $g\notin \overline{\fm_x^{m+n}}$. Therefore, $\widehat{\ord}_x(g)< m+n$, and the claim is complete.
\end{proof}

\subsection{Izumi type estimates}

The propositions in the section concern the following setup, which will arise in the proof of Theorem \ref{thm:izumifamily}. 
 Let $x\in (X,D)$ be an affine klt singularity over an algebraically closed field $\bk$. Fix a projective compactification $X\subset \overline{X}$ and a resolution of singularities $\overline{\pi}: \overline{Y} \to \overline{X}$. Assume there exists a very ample line bundle $L$ on $\overline{Y}$ and the restriction of $\overline{\pi}$ to $X$, denoted $\pi:Y\to X$, is a log resolution of $(X,D, \fm_x)$.

\begin{prop}\label{prop:classicalizumi}
There exists a constant $C_0$ so that the following holds: For any closed point $y \in \pi^{-1}(x)$ and $g\in \cO_{x,X}$, we have
\[
\ord_{y}(\pi^* g) \leq C_0 \cdot  \ord_x(g)
.\]
 Furthermore, if we write $\fm_x 
\cdot \cO_Y= \cO_Y(-\sum_{i=1}^r a_i E_i)$ where each $E_i$ is a prime divisor on $Y$ and $a_i \in \bZ_{>0}$, then there is a formula
for such a constant $C_0$ given in terms of the coefficients of 
$\sum a_i E_i$, the intersection numbers $(E_i \cdot E_j \cdot L^{n-2})$ for $1\leq i,j \leq r$, and the dimension of $X$. \end{prop}

The proposition is a refined version of  \cite[Theorem 3.2]{Li15a}. Its proof  relies on ideas in \cite{BFJ14} and \cite[Appendix II]{Li15a}. 

\begin{proof}

Fix a closed point $y\in \pi^{-1}(x)$ and an element $g\in \cO_{x,X}$. Let $\rho: B_{y} \overline{Y} \to \overline{Y}$ denote the blowup of $\overline{Y}$ at $y$ with exceptional divisor $F_0$. We write $\mu := \pi \circ \rho$ and $F_i$ for the strict transform of $E_i$. 
Consider the divisor $G$ given by the closure of  $\{ \mu^* g=0 \}$ and  write 
\[
G= \sum_{i=0}^{r} b_i F_i + \widetilde{G},\] 
where no $F_i$ lies in the support of $\tilde{G}$. 
Note that $b_0 = \ord_{y}( \pi^*g)$ and 
$b_i: = \ord_{E_i}(g)$ for $1\leq i \leq r$. Since $\overline{Y} \to X$ factors through the normalized blowup of $X$ along $\fm_x$, Proposition \ref{prop:ordhat} implies \[
\min_{1 \leq i \leq r}  \frac{b_i}{a_i} =\min_{1 \leq i \leq r} \frac{ \ord_{E_i}(g)}{a_i}
 \leq \widehat{\ord}_x(g).\] To simplify notation, we set $a := \max\{a_i\}$. 

Our goal will be to find a constant $C$ such that  $b_0 \leq C \cdot b_i$ for all $i \in \{1,\ldots, r\}$. After finding such a $C$, we will have that  
 \[
  \ord_{y}( \pi^*g)  =b_0 \leq  (C/a) \widehat{\ord}_x(g) \leq (C/a) (n+1)\ord_x (g) , \]
  where the last inequality follows from Proposition \ref{prop:ordcomp}. Thus, the desired inequality
 will hold with $C_0 =  (C/a)(n+1)$.

We now proceed to find such a constant $C$. 
 Set $M = \rho^*L - (1/2) F_0$ and note that $M$ is ample \cite[Example 5.1.6]{LazPAG}.  
For each $i \in \{1,\ldots,r\}$, we consider
\[
\sum_{j=0}^r b_j (F_i \cdot F_j \cdot M^{n-2} ) = G \cdot F_i \cdot M^{n-2} - \widetilde{G}\cdot F_i \cdot M^{n-2}
\leq 
G \cdot F_i \cdot M^{n-2} = 0, \]
where the last equality follows from the fact that $G$ is a principal divisor in a neighborhood of $\pi^{-1}(x)$.  
Now, we set \[
c_{ij} := (F_i \cdot F_j \cdot M^{n-2})
. \]
and see 
\[
\sum_{j\neq i } b_j c_{ij} \leq -b_i c_{ii} \leq b_i |c_{ii}|.\]
Note that if $i \neq j$, then $c_{ij} \neq 0$ if and only if $F_i \cap F_j \neq \emptyset$. 
When that is the case, 
\begin{equation}\label{eq:classizumi1}
b_j \leq \frac{|c_{ii} | }{c_{ij}} b_{i}.
\end{equation}
Computing the $c_{ij}$ in terms of intersection numbers on $\overline{Y}$, we find that for $i,j \in \{1,\ldots, r\}$
\[
c_{ij} = \begin{cases}
(E_i \cdot E_j \cdot L^{n-2}) -(1/2)^{n-2} & \text{ if } y  \in E_i \cap E_j \\
(E_i \cdot E_j \cdot L^{n-2})  & \text{ otherwise }
\end{cases} 
.\]
Additionally,
 \[
 c_{0i} = \begin{cases} 
  (1/2)^{n-2} & \text{ if  $y \in E_i$}\\
   0 &\text{ otherwise}
  \end{cases}. 
  \]
  Now, for each $i,j\in \{1,\ldots, r\}$ such that $i\neq j$ and $E_i\cap E_j \neq \emptyset$, we set
  \[
  C_{ij} = \frac{  |( E_i \cdot E_i \cdot L^{n-2})| + (1/2)^{n-2} }
  {(E_i \cdot E_j \cdot L^{n-2} ) - (1/2)^{n-2}}.\] Note that $|c_{ii}|/c_{ij} \leq C_{ij}$. 
  For each $i$, we set
  \[
  C_{i0} =  \frac{  |( E_i \cdot E_i \cdot L^{n-2})| + (1/2)^{n-2} }{(1/2)^{n-2}}
  .\] Similarly, note that $\frac{|c_{ii}|}{c_{0i}} \leq C_{i0}$
 if $y \in E_i$.
 
 Now, set
   $C' = \max \{ 1, C_{ij}, C_{i0}  \}$.
   By our choice of $C'$, if $i,j\in \{0,1,\ldots, r\}$ are distinct and $F_i
   \cap F_j \neq \emptyset$, then  $b_j \leq C' \cdot b_i$. Now, Zariski's Main Theorem implies 
   $\cup F_i$ is connected. Therefore, we set $C = 1+ C'+C'^2+ \cdots C'^r$ and conclude $b_0 \leq C \cdot b_i$ for all $i \in \{1,\ldots,r \}$.
 \end{proof}

\begin{prop}\label{prop:partofizumi}
  There exists a constant $C_1$ such that the following holds:  If $v\in \Val_{x,X}$ satisfies $A_{X,D}(v) < +\infty$ and $y\in c_Y(v)$, then 
\[
v(g) \leq C_1 \cdot A_{X,D}(v) \ord_{y} (\pi^* g)
\]
for all $g\in \cO_{x,X}$. Furthermore, if $K_Y - \pi^*(K_X+D)$ has coefficients $>-1+\epsilon$ with $0 < \epsilon<1$, then the condition holds when $C_1:=1/\epsilon$.
\end{prop}
\begin{proof}
A proof of the statement can be found in the proof  \cite[Theorem 3.1]{Li15a} in the case when $D=0$. The more general statement follows from a similar argument. 
\end{proof}

\subsection{Proofs of Theorems \ref{thm:izumifamily} and \ref{thm:properness}}

\begin{proof}[Proof of Theorem \ref{thm:izumifamily}]

It is sufficient to prove the theorem in the case when both $\cX$ and $T$ are affine. 
We will show that there exists a nonempty open set $U\subset T$ and a constant $K_0 >0$ such that the conclusion of the theorem holds for all $t\in U$. By induction on the dimension of $T$, the proof will be complete.

Fix a (relative) projective compactification $\overline{\pi}: \overline{\cX} \to T$. Denote the ideal sheaf of $\sigma(T)$ in $\cX$ by $\cI_{\sigma(T)}$.
Fix  a  projective resolution of singularities $\overline{\rho} : \overline{\cY}  \to \overline{\cX}$ such that its restriction to $\cX$, denoted $\rho: \cY \to \cX$, is a log resolution of
$(\cX, \cD, \cI_{\sigma(T)})$. Set $\overline{\mu} = \overline{\pi} \circ \overline{\rho}$.  We write
\[
\cI_{\sigma(T)} \cdot \cO_{\cY} = \cO_{\cY} \left( - \sum_{i=1}^k b_i \cE_i \right) \text{ and } K_\cY - \rho^*(K_\cX + \cD)  = \sum_{i=1}^k a_i \cE_i 
\]
where each $\cE_i$ is a prime divisor on $\cY$. We order these prime divisors so that each $\cE_{i}$ dominates $T$ if and only if $1\leq i \leq r$ for some positive integer $r\leq k$.

 By generic smoothness, there exists
 a nonempty open set $U_1 \subset T$ such that $\cY_{\gt} \to \cX_{\gt}$ is a log resolution 
  of $(\cXgt,\cDgt, \fm_{\sgt})$ for all $t\in U_1$ and $\overline{\mu}^{-1}(U_1)\to U_1$ is smooth. Further shrinking $U_1$, we may assume $\cE_{i,\gt} \neq \emptyset$ if and only if $1 \leq i \leq r$.  Let us assume $i \leq r$ for the rest of the proof.
 Now, we have
 \[
 \fm_{\sgt} \cdot \cO_{\cY_{\sgt}} = \cO_\cY \left( -  \sum_{i=1}^r b_i \cE_{i,\gt} \right) 
  \text{ and } K_{\cY_{\gt}} - \rho_{\gt}^*(K_{\cXgt} +\cDgt )  = \sum_{i=1}^r a_i \cE_{i,\gt} \]
  for each $t\in U_1$. 
  Note that the divisors $\cE_{i}\vert_{\gt}$  may have multiple irreducible components. 
  
  Next, we apply \cite[Tag 0551]{SPA} to find and an \'etale morphism $T'\to U_1$, with $T'$ irreducible and so that all irreducible components of the generic fiber of ${\cE'}= \cE_{i} \times_{T} T' \to T'$ are geometrically irreducible.  Denote by $(\cX',\cD',\cY', \cE_i'):=(\cX,\cD,\cY, \cE_i)\times_T T'$, and $\eta'$ the generic points of $T'$. 
 Write
 \[
{\cE'}_{i,\eta'}={\cE'}_{i,1,\eta'} \cup \cdots \cup {\cE'}_{i,m_i,\eta'}
 \]
 for the decomposition of $\cE_{i,\eta'}'$ into irreducible components, and set $\cE_{i,j}$ equal to the closure of $\cE_{i,j,\eta'}'$ in $\cY'$. 
Applying  \cite[Tag 0559]{SPA}, we may find an open subset $U' \subset T'$ so that each divisor ${\cE}_{i,j,\gt}$ is geometrically irreducible for all $t\in U'$. Further shrinking $U'$, we may assume that the divisors $\cE_{i,j,\gt}$ for $1\leq i \leq r$ and $1\leq j \leq m_i$ are distinct. We choose $U\subset T$ to be a nonempty open subset contained in the image of $U'$ in $T$. 
 
 We seek to find a constant $C_0$ such that if
 $t\in U$, then 
 \begin{equation}\label{eq:izumiproof1}
 \ord_{\sgt}(g) \leq C_0 \cdot \ord_{y}(\rho_{\gt}^{*}(g)) 
 \end{equation}
 for all $g\in \cO_{\sgt,\cXgt }$ and $y \in \rho_{\gt}^{-1}(\sgt)$.
 Since $(\cX',\cD',\cY')|_{\overline{t'}}
 \cong (\cX,\cD,\cY)|_{\gt}$ if $t$ is the image of $t'$, it suffices to establish an inequality 
 of the form \eqref{eq:izumiproof1}
on the singularities $\sigma(\overline{t'})\in (\cX'_{\overline{t'}},\cD'_{\overline{t'}})$ for $t' \in U'$. 
 Let $\cL$ be a line bundle on $\overline{\cY}$ such that $\cL_{\gt}$
 is very ample for all $t\in T$, and write $\cL'$ for the pullback of $\cL$ to $\overline{\cY'}$.
 Now, fix $1\leq i_1,i_2\leq r$ such that $b_{i_1},b_{i_2}>0$. For fixed $1\leq j_1\leq m_{i_1}$
 and $1\leq j_2\leq m_{i_2}$, the function that sends $U' \ni t'$ 
 to  $ (\cE_{i_1,j_1, \overline{t'}} \cdot {\cE_{i_2,j_2,\overline{t'} }
 \cdot \cL'_{\overline{t'}}}^{n-2})$ is constant 
 \cite[Lemma VI.2.9]{Kol96}. From our choice of $U'$, we know
 that $\cE_{i,j,\overline{t'}}$ is irreducible for any $t' \in U'$. 
 Therefore, we may apply Proposition \ref{prop:classicalizumi} 
 to find such a constant $C_0$ such that the desired inequality holds for all $t'\in U'$.
  
Next, choose $0<\epsilon<1$ so that  $a_i<1-\epsilon$ for all $1\leq i \leq r$. Set $C_1:= 1/\epsilon$. By 
Proposition \ref{prop:partofizumi}, if $t\in U$ and $v\in \Val_{\sgt,\cXgt}$, then 
 \begin{equation}\label{eq:izumiproof2}
v(g) \leq C_1 \cdot  A_{\cXgt,\cDgt} \ord_{y}( \rho_{\gt}^*g) 
\end{equation}
for all $g\in \cO_{\sgt,\cXgt}$ and $y \in c_{\cY_{\gt}}(v)$. Combining \eqref{eq:izumiproof1} and \eqref{eq:izumiproof2}, we see that the desired inequality holds when $K_0 = C_0 \cdot C_1$. 
\end{proof}

\begin{proof}[Proof of Theorem \ref{thm:properness}]
The theorem follows immediately Theorem \ref{thm:izumifamily} and \cite[Theorem 4.1]{Li15a}, as in the proof of \cite[Theorem 4.3]{Li15a}. 
\end{proof}

\section{Proofs and applications}

\subsection{A convergence result for normalized colengths}\label{sec:hvolcol}

\begin{thm}\label{thm:convncol}
Let $\pi: (\cX,\cD) \to T$ together with a section $\sigma : T \to \cX$ be a $\bQ$-Gorenstein flat family of klt singularities. For every $\epsilon >0$, there exists a constant $c_1>0$ and integer $N$ so  that the following holds: if $t\in T$, then 
\[
\widehat{ \ell_{c,m}}(\sgt,\cXgt,\cDgt)  \leq \nvol(\sgt,\cXgt,\cDgt) +\epsilon \]
for all $m$ divisible by $N$ and $0< c\leq  c_1$. 
\end{thm}

Before beginning the proof of the previous theorem, we record the following statement.

\begin{prop}\label{prop:boundA}
Let $\pi: (\cX,\cD) \to T$ together with a section $\sigma : T \to \cX$ be a $\bQ$-Gorenstein flat family of klt singularities. There exists a constant $A$ so that\[
\nvol(\sgt,\cXgt,\cDgt) = \inf \left\{ \nvol(v) \,\vert\, v\in \Val_{\sgt,\cXgt} \text{ with } v(\fm_{\sgt}) =1 \text{ and } A_{\cXgt,\cDgt}(v) \leq A \right\}\]
for all $t\in T$. 
\end{prop}

\begin{proof}
We first note that there exists a real number $B$ so that $\nvol(\sgt,\cXgt,\cDgt) \leq B$ for all $t\in T$. Indeed, $\nvol(\sgt,\cXgt,\cDgt) \leq \lct( \fm_{\sgt})^n \hs(\fm_{\sgt}) = \lct( \fm_{\st})^n \hs(\fm_{\st})$ and the function that sends $ t\in T$ to $\lct(\fm_{\st})^n \hs(\fm_{\st})$ takes finitely many values by Propositions \ref{lctsemicont} and \ref{prop:hs}. Thus, 
\[
\nvol(\sgt,\cXgt,\cDgt) = \inf \left\{ \nvol(v) \,\vert\, v\in \Val_{ \sgt, \cXgt} \text{ with } v(\fm_{\sgt}) =1 \text{ and } \nvol(v)\leq  B \right\}\]
for all $t\in T$. 

Next, fix a constant  $K_1 \in \bR_{>0}$ satisfying the conclusion of Theorem \ref{thm:properness}. If $v\in \Val_{\sgt,\cXgt}$ satisfies $v(\fm_{\sgt})=1$  and $\nvol(v) \leq B$, then $A_{\cXgt,\cDgt}(v) \leq B/K_2$. Therefore, the proposition holds with $A:= B/K_2$.  
\end{proof}

\begin{proof}[Proof of Theorem \ref{thm:convncol}]

Fix $\varepsilon>0$ and a constant $A \in \bR_{>0}$ satisfying the conclusion of the previous proposition. To simplify notation, set 
\[
W_t = \{ v\in \Val_{\sgt, \cXgt} \, \vert \, v(\fm_{\sgt})=1 \text{ and } A_{(\cXgt,\cDgt)}(v) \leq A \}\]
for each $t\in T$. We proceed by proving the following two claims.\\

\noindent {\bf Claim 1}: There exist constants $c_1\in \bR_{>0}$ and $M_1 \in \bZ_{>0}$ such that the following holds:  if $t\in T$, then
\[
 \widehat{ \ell_{c,m}}(\sgt,\cXgt,\cDgt)  \leq \inf_{v\in W_t} n! \cdot \lct(\fa_m(v))^n \ell ( \cO_{\sgt,\cXgt}/ \fa_m(v)) 
\]
for all $0<c<c_1$, and  $m \geq M_1$. 

Proposition \ref{prop:hs} implies
there exist constants $c_1>0$ and $M_1 \in \bZ_{>0}$ such that 
 \[
\ell ( \cO_{\sgt,\cXgt} / \fm_{\sgt}^m ) \geq m^n \cdot  c_1
\]
for all $t\in T$ and $m \geq M_1$. Now, consider $v\in W_t$ for some $t\in T$. Since $v(\fm_{\sgt})=1$, $\fm_{\sgt}^m \subset \fa_m(v)$ for all $m \in \bZ_{>0}$. Therefore, $\ell ( \cO_{\sgt, \cXgt} / \fa_m(v) ) \geq  \ell ( \cO_{\sgt,\cXgt} / \fm_{\sgt}^m) ) $. The claim now follows from definition of $\widehat{ \ell_{c,m}}(\sgt,\cXgt,\cDgt)$.  \\

\noindent {\bf Claim 2}: There exists $M_2\in \bZ_{>0}$ such that the following holds: if $t\in T$ and $v\in W_t$, then 
\[
n! \cdot  \lct(  \fa_m(v) )^n \cdot \ell( \cO_{\sgt,\cXgt} / \fa_m(v)) \leq \nvol(v) + \varepsilon
\]
for all integers $m$ divisible by $M_2$. 

By Theorem \ref{thm:volcol}, there exists $M_2\in \bZ_{>0}$ such the following holds: If $t\in T$ and $v \in W_t$, then 
\begin{equation}\label{eq:colin}
\frac{ \ell ( \cO_{\sgt,\cXgt }/ \fa_{ m}(v))}{m^n/ n!} 
\leq 
\vol(v) + \varepsilon/A^n 
\end{equation}
for all integers $m$ divisible by $M_2$. Note that 
\[
m \cdot   \lct(  \fa_m(v) ) \leq \lct( \fa_\bullet(v)) \leq A_{ \cXgt,\cDgt}(v).\]
Therefore, multiplying   \eqref{eq:colin} by $(m \cdot   \lct(  \fa_m(v) ))^n$ yields the desired result. \\

We return to the proof of the corollary. Fix constants $M_1$, $M_2$, and $c_1$ satisfying the conclusions of the Claims 1 and 2. Set $M = M_1 \cdot M_2$. Now, if $t\in T$, $m$ is a postive integer divisible by $M$, and $c$ satisfies $0<c<c_1$, then
\begin{align*}
 \widehat{ \ell_{c,m}}(\sgt,\cXgt,\cDgt) & \leq \inf_{v\in W_t} n! \cdot \lct(\fa_m(v))^n \ell ( \cO_{\sgt,\cXgt}/ \fa_m(v)) \\
 		&  \leq  \inf_{v\in W_t} \nvol(v) + \varepsilon \\
		& =  \nvol(\sgt,\cXgt,\cDgt) +\varepsilon ,
\end{align*}
where the first (in)equality follows from Claim 1, the second from Claim 2, and the third from our choice of $A$. 
\end{proof}

\subsection{Proofs}\label{sec_proofs}
The following theorem is a stronger result that implies Theorem \ref{mainthm}.

\begin{thm}\label{weaksc}
Let $\pi:(\cX,\cD)\to T$ together with a section $\sigma: T\to \cX$ be a $\bQ$-Gorenstein flat family of klt singularities over a field $\bk$ of characteristic $0$.
Then the function $t\mapsto \hvol(\sgt,\cXgt,\cDgt)$ on $T$ is lower semicontinuous
with respect to the Zariski topology. 
\end{thm}

\begin{proof}
Let $\cZ_k\to T$ be the $k$-th thickening of the section $\sigma(T)$,
i.e. $\cZ_{k}=\Spec_T(\cO_{\cX}/\cI_{\sigma(T)}^k)$. 
Let $d_k:=\max_{t\in T}\ell(\cO_{\sigma(t),\cX_t}/\fm_{\sigma(t),\cX_t}^k)$.
For any $d\in\bN$,
denote $\cH_{k,d}:=\mathrm{Hilb}_{d}(\cZ_k/T)$.
Since $\cZ_k$ is proper over $T$, we know that
$\cH_{k,d}$ is also proper over $T$. Let $\cH_{k,d}^{\mathrm{n}}$ be the normalization of $\cH_{k,d} $. Denote by $\tau_{k,d}:\cH_{k,d}\to T$. After pulling back the 
universal ideal sheaf on $\cX\times_T\cH_{k,d}$ over $\cH_{k,d}$ to $\cH_{k,d}^{\mathrm{n}}$,
we obtain an ideal sheaf $\fb_{k,d}$ on $\cX\times_T\cH_{k,d}^{\mathrm{n}}$.
Denote by $\pi_{k,d}: (\cX\times_T\cH_{k,d}^{\mathrm{n}}, \cD\times_T \cH_{k,d}^{\mathrm{n}})\to \cH_{k,d}^{\mathrm{n}}$
the projection, then $\pi_{k,d}$ provides a $\bQ$-Gorenstein flat family of klt pairs.

Following the notation of Proposition \ref{fieldext}, assume $h$ is 
scheme-theoretic point of $\cH_{k,d}^{\mathrm{n}}$ lying over 
$[\fb]\in\cH_{k,d}$. Denote by $t=\tau_{k,d}([\fb])\in T$. By 
construction, the ideal sheaf $\fb_{k,d,h}$ on $\cX\times_{T}
\Spec(\kappa(h))$ is the pull back of  $\fb$ under the flat base 
change $\Spec(\kappa(h))\to\Spec(\kappa([\fb]))$. Hence 
$$\lct((\cX,\cD)\times_{T}\Spec(\kappa(h));\fb_{k,d,h})=
\lct((\cX,\cD)\times_{T}\Spec(\kappa([\fb]));\fb).$$
For simplicity, we abbreviate the above equation to $\lct(\fb_{k,d,h})=\lct(\fb)$.
Applying Proposition \ref{lctsemicont} to the family $\pi_{k,d}$ and the ideal $\fb_{k,d}$ implies that the function $\Phi^{\mathrm{n}}:\cH_{k,d}^{\mathrm{n}}\to\bR_{>0}$ defined as
$\Phi^{\mathrm{n}}(h):=\lct(\fb_{k,d,h})$ is constructible and lower semicontinuous with respect to the Zariski topology on $\cH_{k,d}^{\mathrm{n}}$. Since $\lct(\fb_{k,d,h})=\lct(\fb)$, $\Phi^{\mathrm{n}}$ descends to a function $\Phi$ on $\cH_{k,d}$ as $\Phi([\fb]):=\lct(\fb)$. Since $\cH_{k,d}$ is proper over $T$, we know that the function $\phi:T\to \bR_{>0}$ defined as
\[
\phi( t):=n!\cdot\min_{\substack{ck^n\leq d\leq d_k\\ [\fb]\in\tau_{k,d}^{-1}(t)}}\Phi([\fb])^n
\]
is constructible and lower semicontinuous with respect to the 
Zariski topology on $T$. Then Proposition \ref{fieldext} implies
$\phi(t)=\widehat{\ell_{c,k}}(\sigma(\gt), \cX_{\gt},D_{\gt})$. Thus 
we conclude that $t\mapsto\widehat{\ell_{c,k}}(\sigma(\gt), \cX_{\gt},D_{\gt})$ 
is constructible and lower semicontinuous with respect to the
Zariski topology on $T$.

Let us fix $\epsilon>0$ and a scheme-theoretic point
$o\in T$. By Theorem \ref{thm:convncol}, there exist $c_1>0$ and $N\in \bN$
such that
\begin{equation}\label{eq:mainproof1}
 \hvol(\sgt,\cXgt,\cDgt)\geq \widehat{\ell_{c,k}}(\sgt,\cXgt,\cDgt)-\frac{\epsilon}{2}
\end{equation}
for any $t\in T$, $k$ divisible by $N$ and $0<c\leq c_1$.
Since  $t\mapsto\widehat{\ell_{c,k}}(\sigma(\gt), \cX_{\gt},\cD_{\gt})$
is constructibly lower semicontinuous on $T$, there exists a Zariski open
neighborhood $U$ of $o$ such that
\begin{equation}\label{eq:mainproof2}
 \widehat{\ell_{c,k}}(\sgt,\cXgt,\cDgt)\geq \widehat{\ell_{c,k}}(\sigma(\overline{o}),\cX_{\overline{o}},\cD_{\overline{o}})
 \quad\textrm{for any }t\in U.
\end{equation}
By Theorem \ref{hvolcolength}, there exist $c_0>0$ and $N_0\in \bN$
such that 
\begin{equation}\label{eq:mainproof3}
\widehat{\ell_{c,k}}(\sigma(\overline{o}),\cX_{\overline{o}},\cD_{\overline{o}})
\geq \hvol(\sigma(\overline{o}),\cX_{\overline{o}},\cD_{\overline{o}})-\frac{\epsilon}{2}
\end{equation}
for any $0<c\leq c_0$ and any $k\geq N_0$. Let us 
choose $c=\min\{c_0,c_1\}$ and $k=N\cdot N_0$. Then combining \eqref{eq:mainproof1},
\eqref{eq:mainproof2} and \eqref{eq:mainproof3} yields that
\[
 \hvol(\sgt,\cXgt,\cDgt)\geq\hvol(\sigma(\overline{o}),\cX_{\overline{o}},\cD_{\overline{o}})-\epsilon
 \quad\textrm{for any }t\in U.
\]
The proof is finished.
\end{proof}

The following theorem is a stronger result that implies
Theorem \ref{openk}.

\begin{thm}\label{openkgen}
Let $\varphi:(\cY,\cE)\to T$ be a $\bQ$-Gorenstein flat 
family of log Fano pairs over a field $\bk$ of
characteristic $0$. Assume that some geometric fiber 
$(\cY_{\overline{o}},\cE_{\overline{o}})$ is log K-semistable 
for a point $o\in T$. Then
\begin{enumerate}
\item There exists an intersection
$U$ of countably many Zariski open neighborhoods of $o$,
such that $(\cY_{\gt},\cE_{\gt})$ is log
K-semistable for any  point $t\in T$. If, in addition,
$\bk=\bar{\bk}$ is uncountable, then $(\cY_{t},\cE_{t})$ 
is log K-semistable for a very general closed point $t\in T$.
\item The geometrically log K-semistable locus $$T^{\textrm{K-ss}}:=\{t\in T\colon (\cY_{\gt},\cE_{\gt})\textrm{ is log K-semistable}\}$$ is stable under generalization.
\end{enumerate}
\end{thm}

\begin{proof}
(1) For $r\in\bN$ satisfying $\cL=-r(K_{\cY/T}+\cE)$
is Cartier, we define the \emph{relative affine cone} $\cX$ of $(\cY,\cL)$ by
\[
 \cX:=\Spec_{T}\oplus_{m\geq 0}\varphi_*(\cL^{\otimes m}).
\]
Assume $r$ is sufficiently large, then it is easy to see that $\varphi_*(\cL^{\otimes m})$
is locally free on $T$ for all $m\in\bN$. Thus we have 
$\cX_t\cong \Spec\oplus_{m\geq 0}H^0(\cY_t,\cL_t^{\otimes m}):=C(\cY_t,\cL_t)$.
Let $\cD$ be the $\bQ$-divisor on $\cX$ corresponding to $\cE$. 
By \cite[Section 3.1]{Kol13}, the projection $\pi:(\cX,\cD)\to T$ together
with the section of cone vertices $\sigma:T\to\cX$ is
a $\bQ$-Gorenstein flat family of klt singularities.

Since $(\cY_{\overline{o}},\cE_{\overline{o}})$ is K-semistable, Theorem \ref{ksscone} implies
$$\hvol(\sigma(\overline{o}),\cX_{\overline{o}}, \cD_{\overline{o}})=r^{-1}(-K_{\cY_{\overline{o}}}-\cE_{\overline{o}})^{n-1}.$$
Then by Theorem \ref{weaksc}, there exists an intersection $U$ of countably many Zariski open neighborhoods of $o$, such that $\hvol(\sigma(\gt),\cX_{\gt}, \cD_{\gt})\geq \hvol(\sigma(\overline{o}),\cX_{\overline{o}}, \cD_{\overline{o}})$ for any $t\in U$. Since the global volumes of log Fano pairs are constant in $\bQ$-Gorenstein flat families, we have
\[
\hvol(\sigma(\gt),\cX_{\gt}, \cD_{\gt})\geq \hvol(\sigma(\overline{o}),\cX_{\overline{o}}, \cD_{\overline{o}})=
r^{-1}(-K_{\cY_{\overline{o}}}-\cE_{\overline{o}})^{n-1}=r^{-1}(-K_{\cY_{\gt}}-\cE_{\gt})^{n-1}.
\]
Then Theorem \ref{ksscone} implies that $(\cY_{\gt},\cE_{\gt})$ is K-semistable for any $t\in U$.
\medskip

(2) Let $o\in T^{\textrm{K-ss}}$ be a scheme-theoretic point.
Then by (1) there exists countably many Zariski open neighborhoods $U_i$ of $o$ such that $\cap_i U_i\subset T^{\textrm{K-ss}}$.
If $t$ is a generalization of $o$, then $t$ belongs to all Zariski open neighborhoods of $o$, so $t\in T^{\textrm{K-ss}}$.
\end{proof}

\begin{proof}[Proof of Theorem \ref{openk}]
It is clear that (1) and (2) follows from Theorem \ref{openkgen}.
For (3), the constructibility of normalized volumes implies that
the set $U$ in the proof of Theorem \ref{openkgen} (1) can be chosen
as a Zariski open neighborhood of $o$. Then the same argument in the proof of Theorem \ref{openkgen} (1)
works.
\end{proof}

The following corollary is a stronger result that implies
Corollary \ref{specialdeg}.

\begin{cor}
 Let $\pi:(\cY,\cE)\to T$ be a $\bQ$-Gorenstein family of complex log Fano pairs. Assume that $\pi$ is isotrivial over a Zariski open subset $U\subset T$, and $(\cY_o,\cE_o)$ is log K-semistable for a closed point $o\in T\setminus U$. Then $(\cY_t,\cE_t)$ is log K-semistable for any $t\in U$.
\end{cor}

\begin{proof}
Since $(\cY_o,\cE_o)$ is log K-semistable, Theorem \ref{openkgen} implies that $(\cY_t,\cE_t)$ is log K-semistable for very general closed point $t\in T$. Hence there exists (hence any) $t\in U$ such that $(\cY_t,\cE_t)$ is log K-semistable.
\end{proof}

\begin{rem}
 If the ACC of normalized volumes (in bounded families)
 were true, then Conjecture \ref{mainconj} follows by applying Theorem
 \ref{mainthm}. Moreover, we suspect that a much stronger result 
 on discreteness of normalized volumes away from $0$ (see also \cite[Question 4.3]{LiuX17}) might be true, but we don't
 have much evidence yet.
\end{rem}

\subsection{Applications}\label{sec_appl}
In this section we present applications of Theorem \ref{mainthm}. The
following theorem generalizes the inequality part of \cite[Theorem A.4]{LiuX17}.
\begin{thm}\label{maxhvol}
Let $x\in (X,D)$ be a complex klt singularity of dimension $n$. Let
$a$ be the largest coefficient of components of $D$ containing $x$. Then
$\hvol(x,X,D)\leq (1-a)n^n$.
\end{thm}

\begin{proof}
Suppose $D_i$ is the component of $D$ containing $x$ with
coefficient $D$. Let $D_i^{\mathrm{n}}$ be the normalization
of $D_i$. By applying Theorem \ref{mainthm} to $\mathrm{pr}_2:
(X\times D_i^{\mathrm{n}}, D\times D_i^{\mathrm{n}})\to 
D_i^{\mathrm{n}}$ together with the natural diagonal section
$\sigma:D_i^{\mathrm{n}}\to X\times D_i^{\mathrm{n}}$, we have that
$\hvol(x,X,D)\leq\hvol(y, X, D)$ for a very general closed point 
$y\in D_i$. We may pick $y$ to be a smooth point in 
both $X$ and $D$, then $\hvol(x,X,D)\leq \hvol(0,\bA^n,a \bA^{n-1})$ 
where $\bA^{n-1}$ is a coordinate hyperplane of $\bA^n$. Let us take
local coordinates $(z_1,\cdots,z_n)$ of $\bA^n$ such that $\bA^{n-1}=V(z_1)$.
Then the monomial valuation $v_a$ on $\bA^n$ with weights $((1-a)^{-1},1,\cdots,1)$
satisfies $A_{\bA^n}(v)=\frac{1}{1-a}+(n-1)$, $\ord_{v_a}(\bA^{n-1})=\frac{1}{1-a}$
and $\vol(v_a)=(1-a)$. Hence
\[
 \hvol(x,X,D)\leq \hvol_{0,(\bA^n,a\bA^{n-1})}(v_a)=(A_{\bA^n}(v)-a\ord_{v_a}(\bA^{n-1}))^n
 \cdot\vol(v_a)=(1-a)n^n.
\]
The proof is finished.
\end{proof}

\begin{thm}\label{thm:kltvar}
Let $(X,D)$ be a klt pair over $\bC$. Then
\begin{enumerate}
\item The function $x\mapsto \hvol(x,X,D)$ on $X(\bC)$ is lower
semicontinuous with respect to the Zariski topology.
\item Let $Z$ be an irreducible subvariety of $X$. Then for a very general closed point $z\in Z$ we have
\[
\hvol(z,X,D)=\sup_{x\in Z}\hvol(x, X,D).
\]
In particular, there exists a countable intersection $U$ of non-empty Zariski open subsets of $Z$ such that $\hvol(\cdot, X,D)|_U$ is constant. 
\end{enumerate}
\end{thm}

\begin{proof}
Part (1) follows quickly by applying Theorem \ref{mainthm}
to $\mathrm{pr}_2:(X\times X, D\times X)\to X$
together with the diagonal section $\sigma:X\to X\times X$. 
For part (2), denote by $Z^{\mathrm{n}}$ the normalization of $Z$.
Then the proof follows quickly by applying Theorem \ref{mainthm} to $\mathrm{pr}_2:(X\times Z^{\mathrm{n}},D\times Z^{\mathrm{n}})\to Z^{\mathrm{n}}$ together with the natural diagonal section $\sigma:Z^{\mathrm{n}}\to X\times Z^{\mathrm{n}}$.
\end{proof}

Next we study the case when $X$ is a Gromov-Hausdorff limit
of K\"ahler-Einstein Fano manifolds. Note that the function
$x\mapsto\hvol(x,X)=n^n\cdot\Theta(x,X)$ is lower semicontinuous
with respect to the Euclidean topology on $X$ by \cite{SS17, LX17}.
The following corollary improves this result and follows easily from part (1) of Theorem \ref{thm:kltvar}.

\begin{cor}\label{cor:sctheta}
Let $X$ be a Gromov-Hausdorff limit of K\"ahler-Einstein Fano manifolds. Then the function $x\mapsto\hvol(x,X)=n^n\cdot\Theta(x,X)$ on $X(\bC)$ is lower semicontinuous with respect to the Zariski topology.
\end{cor}

The following theorem partially generalizes \cite[Lemma 3.3 and Proposition 3.10]{SS17}.
\begin{thm}\label{ghlimit}
Let $X$ be a Gromov-Hausdorff limit of K\"ahler-Einstein Fano manifolds. Let $x\in X$ be any closed point. Then for any finite quasi-\'etale morphism of singularities $\pi:(y\in Y)\to(x\in X)$, we have 
$\deg(\pi)\leq \Theta(x,X)^{-1}$. In particular, we have
\begin{enumerate}
    \item $|\hat{\pi}_1^{\mathrm{loc}}(X,x)|\leq
    \Theta(x,X)^{-1}$.
    \item For any $\bQ$-Cartier Weil divisor $L$ on $X$, we have
    $\mathrm{ind}(x,L)\leq\Theta(x,X)^{-1}$ where $\mathrm{ind}(x,L)$ denotes the
    Cartier index of $L$ at $x$.
\end{enumerate}
\end{thm}
\begin{proof}
 By \cite[Theorem 1.7]{LX17}, the finite degree formula holds for $\pi$, i.e.
 $\hvol(y,Y)=\deg(\pi)\cdot\hvol(x,X)$. Since $\hvol(y,Y)\leq n^n$ by 
 \cite[Theorem A.4]{LiuX17} or Theorem \ref{maxhvol} and $\hvol(x,X)
 =n^n\cdot\Theta(x,X)$ by \cite[Corollary 3.7]{LX17}, we have
 $\deg(\pi)\leq n^n/\hvol(x,X)=\Theta(x,X)^{-1}$.
\end{proof}

\begin{rem}\label{r_localpi1}
If the finite degree formula \cite[Conjecture 4.1]{LiuX17} were
true for any klt singularity, then clearly $\deg(\pi)\leq n^n/\hvol(x,X)$ holds for any 
finite quasi-\'etale morphism $\pi: (y,Y)\to (x,X)$ between $n$-dimensional klt singularities. 
In particular, we would get an effective upper bound 
$|\hat{\pi}_1^{\mathrm{loc}}(X,x)|\leq n^n/\hvol(x,X)$ where
$\hat{\pi}_1^{\mathrm{loc}}(X,x)$ is known to be finite
by \cite{Xu14, Bgo17} (see \cite[Theorem 1.5]{LiuX17} for a
partial result in dimension $3$).
\end{rem}

\begin{thm}
Let $V$ be a K-semistable complex $\bQ$-Fano variety of dimension $(n-1)$. Let $q$ be the largest integer such that there exists a Weil divisor $L$ satisfying $-K_V\sim_{\bQ}qL$. Then 
\[
q\cdot(-K_V)^{n-1}\leq n^n.
\]
\end{thm}

\begin{proof}
 Consider the orbifold cone $X:=C(V,L)=\Spec(\oplus_{m\geq 0}H^0(V,
 \cO_V(\lfloor m L\rfloor))$ with the cone vertex $x\in X$.
 Let $\tilde{X}:=\Spec_V\oplus_{m\geq 0}\cO_V(\lfloor mL\rfloor)$ 
 be the partial resolution of $X$ with exceptional divisor $V_0$.
 Then by \cite[40-42]{Kol04}, $x\in X$ is a klt singularity, and $(V_0,0)\cong (V,0)$
 is a K-semistable Koll\'ar component over $x\in X$. Hence \cite[Theorem A]{LX16} implies
 that $\ord_{V_0}$ minimizes $\hvol_{x,X}$. By \cite[40-42]{Kol04}
 we have $A_X(\ord_{V_0})=q$, $\vol(\ord_{V_0})=(L^{n-1})$. Hence
 $$\hvol(x,X)=A_X(\ord_{V_0})^n\vol(\ord_{V_0})=q^n(L^{n-1})=q(-K_V)^{n-1},$$
 and the proof is finished since $\hvol(x,X)\leq n^n$ by \cite[Theorem A.4]{LiuX17}
 or Theorem \ref{maxhvol}.
 \end{proof}

\appendix

\section{Asymptotic lattice points counting in convex bodies}
\label{app:lattice}

In this appendix, we will prove the following proposition.

\begin{prop}\label{convexgeo}
For any positive number $\epsilon$, there exists
$k_0=k_0(\epsilon, n)$ such that for any closed convex body
$\Delta\subset [0,1]^n$ and any integer $k\geq k_0$, we have
\begin{equation}\label{ineq_convexgeo}
\left|\frac{\#(k\Delta\cap\bZ^n)}{k^n}-\vol(\Delta)\right|\leq\epsilon.
\end{equation}
\end{prop}

\begin{proof}
We do induction on dimensions. If $n=1$, then $k\Delta$ is a closed
interval of length $k\vol(\Delta)$, hence we know
\[
 k\vol(\Delta)-1\leq\#(k\Delta\cap\bZ)\leq k\vol(\Delta)+1.
\]
So \eqref{ineq_convexgeo} holds for $k_0=\lceil 1/\epsilon\rceil$.

Next, assume that the proposition is true for dimension $n-1$.
Denote by $(x_1,\cdots,x_n)$ the
coordinates of $\bR^n$. Let $\Delta_t:=\Delta\cap\{x_n=t\}$ be
the sectional convex body in $[0,1]^{n-1}$. Let $[t_{-},t_{+}]$ be
the image of $\Delta$ under the projection onto the last coordinate.
Then we know that $\vol(\Delta)=\int_{t_{-}}^{t_+}\vol(\Delta_t)dt$.
By induction hypothesis, there exists $k_1\in \bN$ such that $$
\vol(\Delta_t)-\frac{\epsilon}{3}\leq \frac{\#(k\Delta_t\cap\bZ^{n-1})}{k^{n-1}}
\leq \vol(\Delta_t)+\frac{\epsilon}{3} \quad \textrm{for any }k\geq k_1.
$$
It is clear that
$$\#(k\Delta\cap\bZ^n)=\sum_{t\in[t_-,t_+]\cap\frac{1}{k}\bZ}\#(k\Delta_t\cap\bZ^{n-1}),$$
so for any $k\geq k_1$ we have
\begin{equation}\label{ineq_cg1}
 \left|\#(k\Delta\cap\bZ^n) -k^{n-1}\cdot \sum_{t\in[t_-,t_+]\cap\frac{1}{k}\bZ}\vol(\Delta_t)\right|
 \leq \frac{\epsilon}{3} k^{n-1}\cdot\#([t_-,t_+]\cap\frac{1}{k}\bZ)\leq \frac{2\epsilon}{3} k^n.
\end{equation}

Next, we know that the function $t\mapsto\vol(\Delta_t)^{1/{(n-1)}}$ is concave on $[t_-,t_+]$ by the Brunn-Minkowski theorem.
In particular, we can find $t_0\in [t_-,t_+]$ such that $g(t):=\vol(\Delta_t)$ reaches 
its maximum at $t=t_0$. Hence $g$ is increasing on $[t_-,t_0]$
and decreasing on $[t_0,t_+]$.
Then applying Proposition \ref{integration} to $g|_{[t_-,t_0]}$
and $g|_{[t_0,t_+]}$ respectively yields
\begin{align*}
 \left|\int_{t_-}^{t_0}\vol(\Delta_t)dt-\frac{1}{k}\sum_{t\in[t_-,t_0]\cap\frac{1}{k}\bZ}\vol(\Delta_t)\right|
\leq \frac{2}{k},\\
\left|\int_{t_0}^{t_+}\vol(\Delta_t)dt-\frac{1}{k}\sum_{t\in[t_0,t_+]\cap\frac{1}{k}\bZ}\vol(\Delta_t)\right|
\leq \frac{2}{k}.
\end{align*}
Since $0\leq \vol(\Delta_{t_0})\leq 1$, we have
\begin{equation}\label{ineq_cg2}
 \left|\int_{t_-}^{t_+}\vol(\Delta_t)dt-\frac{1}{k}\sum_{t\in[t_+,t_-]\cap\frac{1}{k}\bZ}\vol(\Delta_t)\right|
 \leq \frac{5}{k}.
\end{equation}
Therefore, by setting $k_0=\max(k_1,\lceil 15/\epsilon\rceil)$, the 
inequality \eqref{ineq_convexgeo}
follows easily by combining \eqref{ineq_cg1} and \eqref{ineq_cg2}.
\end{proof}

\begin{prop}\label{integration}
 For any monotonic function $g:[a,b]\to [0,1]$ and any $k\in\bN$, we have
 \[
  \left|\int_{a}^b g(s) ds-\frac{1}{k}\sum_{t\in[a,b]\cap\frac{1}{k}\bZ}g(t)\right|\leq\frac{2}{k}.
 \]
\end{prop}

\begin{proof}
 We may assume that $g$ is an increasing function.
 Denote $a_k:=\frac{\lceil ka\rceil}{k}$ and $b_k:=\frac{\lfloor kb\rfloor}{k}$,
 so $[a,b]\cap\frac{1}{k}\bZ=[a_k,b_k]\cap\frac{1}{k}\bZ$.
 Since $\int_{t-1/k}^{t} g(s)ds\leq g(t)/k$ whenever $t\in [a_k+1/k,b_k]$,
 we have 
 \[
  \int_{a_k}^{b_k}g(s)ds \leq \frac{1}{k}\sum_{t\in[a_k+1/k,b_k]\cap\frac{1}{k}\bZ}g(t)
  \leq \frac{1}{k}\sum_{t\in[a,b]\cap\frac{1}{k}\bZ}g(t),
 \]
Similarly, $\int_{t}^{t+1/k} g(s)ds\geq g(t)/k$ for any $t\in [a_k, 
b_k-1/k]$, we have
 \[
  \int_{a_k}^{b_k}g(s)ds\geq\frac{1}{k}\sum_{t\in[a_k,b_k-1/k]
  \cap\frac{1}{k}\bZ}g(t)\geq\frac{1}{k}\sum_{t\in[a,b]
  \cap\frac{1}{k}\bZ}g(t)-\frac{1}{k}.
 \]
 It is clear that $a_k\in[a, a+1/k]$ and $b_k\in[b-1/k, b]$,
 so we have
 \[
  \int_{a_k}^{b_k}g(s)ds\geq \int_a^bg(s)ds-\frac{2}{k},\qquad
  \int_{a_k}^{b_k}g(s)ds\leq \int_a^bg(s)ds.
 \]
 As a result, we have
 \[
 \frac{1}{k}\sum_{t\in[a,b]  \cap\frac{1}{k}\bZ}g(t)-\frac{1}
 {k}\leq \int_{a}^b g(s)ds\leq \frac{1}{k}\sum_{t\in[a,b]
  \cap\frac{1}{k}\bZ}g(t)+\frac{2}{k}
 \]
\end{proof}

\section{Families of Ideals and the Hilbert--Samuel Function}
\label{app:HS}

The following proposition concerns the behavior of the Hilbert--Samuel function along a family of ideals. 
The statement is not new. The proof we give follows arguments found found in \cite{FM00}. 

\begin{defn}
If $(R,\fm)$ is a local ring and $I$ is an $\fm$-primary ideal, 
then the Hilbert--Samuel function of $I$, denoted $\HS_I: \bN \to \bN$, is  given by $\HS_I(m):= \ell_R( R/ I^m)$.
Note that $\hs(I)= \lim_{n \to \infty} \HS_I(m)/ m^n$, where $n = \dim(R)$. 
 \end{defn}

\begin{prop}\label{prop:hs}
Let $\pi: \cX \to T$ be a morphism of finite type $\bk$-schemes. Assume $T$ is integral and $\pi$ has a section $\sigma: T\to \cX$. If  $\fa \subset \cO_{\cX}$ is an ideal and $\fa_t = \fa \cdot \cO_{\cX_{\st }}$ is $\fm_{\st}$-primary for all $t\in T$, then $T$ has a filtration\[
\emptyset = T_0  \subset T_{1} \subset \cdots T_1 \subset T_m = T \]
 such that for every $1\leq i \leq m$, $T_i$ is closed in $T$ and the function $T_i \setminus T_{i-1}\ni t  \mapsto \HS_{\fa_t}$ is constant. 
\end{prop}

\begin{proof}
To prove the result, it is sufficient to show that there exists a nonempty open set $U\subset T$ such that $\HS_{\fa_t}$ is constant for all $t\in U$. 
We proceed to find such a set $U$. 

For each $t\in T$, we have $
\HS_{\fa_t }(m) 
= \sum_{i=0}^{m-1} \ell ( \fa_t^{i} / \fa_t^{i+1})$. 
Therefore, we consider the 
finitely generated $\cO_X$-algebra $\gr_\fa := \oplus_{i\geq 0 } \fa^i/ \fa^{i+1}$. By generic flatness, we may choose a nonempty open set $U\subset T$ such that both $\cO_X\vert_{\pi^{-1}(U)}$ and  $\gr_\fa\vert_{\pi^{-1}(U)}$ are flat over $U$. 

For each $i \in \bN$, the function $U \ni t \mapsto  \dim_{\kappa(t)}(\fa^i /\fa^{i+1}\vert_t)$ is  constant, since each $\fa^i /\fa^{i+1}$ is flat over $U$ and  $\fa^i /\fa^{i+1}\vert_t$ has zero dimensional support for each $t\in U$.  Since $\kappa(t) \simeq \cO_{\st,\cX}/ \fm_{\st}$,  $\dim_{\kappa(t)}(\fa^i/ \fa^{i+1}\vert_t ) = \ell(\fa^i/ \fa^{i+1} \vert_t )$ for all $t\in T$. Furthermore,  Lemma \ref{lem:grflat},  proved below, implies $ \fa^i/ \fa^{i+1} \vert_t  =  \fa_t^i / \fa_t^{i+1}$ for all $t\in U$. Therefore, $U \ni t \mapsto  \ell ( \fa_t^{i} / \fa_t^{i+1})$ is  constant, and the proof is complete. 
\end{proof}

Before stating the following lemma, we introduce some notation. Let $A$ be a ring, $I\subset A$ an ideal, and $M$ an $A$-module. We set
\[
\gr_I (M) := \bigoplus_{m\geq0} \frac{I^m M}{I^{m+1} M}.\]

\begin{lem}\label{lem:grflat}
Let $B\to A$ be a morphism of  rings, $I\subset A$ an ideal, and $M \in \Mod(A)$. 
If $\gr_I M$ and $M$ are both flat over $B$, then for any $N \in \Mod(B)$
\[
(\gr_I M) \otimes_B N \simeq \gr_I (M \otimes_B N).
\]
\end{lem}
\begin{proof} We follow the argument given in \cite{FM00}. 
Consider the surjective map $\alpha_m : (I^{m} M ) \otimes_B N \to I^m ( M \otimes_B N)$.  We claim that, for each $m \in \bZ_{>0}$, $\alpha_m$ is injective and $I^m M$ is flat over $B$. 

In order to prove the claim, we induct on $m$. The claim holds when $m=0$, since $\alpha_0$ is clearly an isomorphism and $M$ is flat over $B$ by assumption. Next, consider the exact sequence
\[0\to 
I^{m+1} M \to I^m M \to I^m M/ I^{m+1} M \to 0\]
and assume the claim holds for a positive integer $m$. 
Since $I^m M$ and $ I^m M/ I^{m+1} M$ are flat over $B$, so is $I^{m+1}M$. By the flatness of $ I^m M/ I^{m+1} M$, we may tensor by $N$ to get an exact sequence 
\[
0\to 
I^{m+1} M  \otimes_B N \to I^m M \otimes_B N \to I^m M/ I^{m+1} M \otimes_B N \to 0
.\]
By the above exact sequence, the injectivity of $\alpha_m$ implies the injectivity of  $\alpha_{m+1}$. Now that the claim has been proven, the lemma follows from applying the claim to the previous short exact sequence. 
\end{proof}

\end{document}